\documentclass[largeformat]{interact}

\usepackage{epstopdf}
\usepackage[caption=false]{subfig}

\usepackage[natbibapa,nodoi]{apacite}
\theoremstyle{plain}
\newtheorem{thm}{Theorem}[section]
\newtheorem{lem}[thm]{Lemma}
\newtheorem{cor}[thm]{Corollary}

\theoremstyle{definition}
\newtheorem{defn}[thm]{Definition}
\newtheorem{example}[thm]{Example}

\theoremstyle{remark}
\newtheorem{rem}{Remark}

\usepackage{color}

\begin{document}


\title{On the Exact Linearization and Control of Flat Discrete-time Systems}

\author{
\name{Bernd Kolar\textsuperscript{a}\thanks{CONTACT Bernd Kolar. Email: bernd\underline{ }kolar@ifac-mail.org\\ The second and the third author have been supported by the Austrian Science Fund (FWF) under grant number P 32151.}, Johannes Diwold\textsuperscript{b}, Conrad Gstöttner\textsuperscript{b}, and Markus Schöberl\textsuperscript{b}}
\affil{\textsuperscript{a}Magna Powertrain Engineering Center Steyr GmbH \& Co KG, St. Valentin, Austria;
\textsuperscript{b}Institute of Automatic Control and Control Systems Technology, Johannes Kepler
University, Linz, Austria}
}

\maketitle

\begin{abstract}
The paper addresses the exact linearization of flat nonlinear discrete-time systems by generalized static or dynamic feedbacks which may also depend on forward-shifts of the new input. We first investigate the question which forward-shifts of a given flat output can be chosen in principle as a new input, and subsequently how to actually introduce the new input by a suitable feedback. With respect to the choice of a feasible input, easily verifiable conditions are derived. Introducing such a new input requires a feedback which may in general depend not only on this new input itself but also on its forward-shifts. This is similar to the continuous-time case, where feedbacks which depend on time derivatives of the closed-loop input -- and in particular quasi-static ones -- have already been used successfully for the exact linearization of flat systems since the nineties of the last century. For systems with a flat output that does not depend on forward-shifts of the input, it is shown how to systematically construct a new input such that the total number of the corresponding forward-shifts of the flat output is minimal. Furthermore, it is shown that in this case the calculation of a linearizing feedback is particularly simple, and the subsequent design of a discrete-time flatness-based tracking control is discussed. The presented theory is illustrated by the discretized models of a wheeled mobile robot and a 3DOF helicopter.
\end{abstract}


\section{Introduction}

The concept of flatness has been introduced in the 1990s by Fliess,
L\'{e}vine, Martin and Rouchon for nonlinear continuous-time systems
(see e.g. \cite{FliessLevineMartinRouchon:1995} and \cite{FliessLevineMartinRouchon:1999}).
Roughly speaking, a continuous-time system is flat if all system variables
can be parameterized by a flat output and its time derivatives, which
in turn depends on the system variables and their time derivatives.
In other words, there exists a one-to-one correspondence between the
trajectories of a flat system and the trajectories of a trivial system.
Since these properties allow an elegant solution to motion planning
problems and a systematic design of tracking controllers, flatness
belongs doubtlessly to the most popular nonlinear control concepts.

For the practical implementation of a flatness-based control, it is
important to evaluate the continuous-time control law at a sufficiently
high sampling rate. If this is not possible, it can be advantageous
to design the controller directly for a suitable discretization of
the continuous-time system, see e.g. \cite{DiwoldKolarSchoberl:2022}.
Thus, transferring the flatness concept to discrete-time systems is
both interesting from a theoretical and an application point of view.
In fact, there are two equally obvious possibilities: The first one
is to replace the time derivatives of the continuous-time definition
by forward-shifts. This point of view has been adopted e.g. in \cite{KaldmaeKotta:2013},
\cite{Sira-RamirezAgrawal:2004}, or \cite{KolarKaldmaeSchoberlKottaSchlacher:2016},
and is consistent with the linearizability by a discrete-time endogenous
dynamic feedback as it is defined in \cite{Aranda-BricaireMoog:2008}.
The second approach is based on the one-to-one correspondence of the
system trajectories to the trajectories of a trivial system. In contrast
to the first approach, in this case the flat output may also depend
on backward-shifts of the system variables, see \cite{DiwoldKolarSchoberl:2020}
or \cite{GuillotMillerioux:2020}. Since it is more general, within
the present paper we consider the second approach, and refer to the
first one, which is included as a special case, as forward-flatness.

For flat continuous-time systems, the problem of exact linearization
by different types of feedback as well as a subsequent tracking control
design have already been studied extensively since the nineties of
the last century. The simplest approach is to perform an exact linearization
such that the new input corresponds to the highest time derivatives
of the flat output that appear in the parameterization of the system
variables. This can always be achieved by a classical (endogenous)
dynamic feedback, see e.g. \cite{FliessLevineMartinRouchon:1999}.
However, it is well-known that it can be advantageous to choose appropriate
lower-order time derivatives of the flat output as a new input. In
this case, the required feedback depends also on time derivatives
of the new input, and is hence not a dynamic feedback in the classical
sense -- see \cite{DelaleauRudolph:1998}, \cite{Rudolph:2021},
or \cite{GstoettnerKolarSchoberl:2021}. An important advantage of
such an exact linearization is that it allows e.g. to achieve a lower-order
error dynamics for a subsequently designed tracking control.

For flat discrete-time systems, it has been shown in \cite{DiwoldKolarSchoberl:2020}
that -- similar to the continuous-time case -- they can be exactly
linearized by a dynamic feedback such that the new input of the closed-loop
system corresponds to the highest forward-shifts of the flat output
that appear in the parameterization of the system variables. Furthermore,
it has been demonstrated in \cite{DiwoldKolarSchoberl:2022} for the
practical example of a gantry crane that also an exact linearization
which introduces lower-order forward-shifts of the flat output as
a new input is in principle applicable. This allows to achieve a lower-order
tracking error dynamics, but requires a feedback which depends also
on forward-shifts of the new input. Since in the particular case of
the gantry crane the required feedback does not have a dynamic part,
it is a discrete-time quasi-static feedback -- see e.g. \cite{ArandaKotta:2001},
where such a type of feedback is used in a different context. However,
according to the authors' best knowledge, a thorough theoretical analysis
of the problem of exact linearization and tracking control design
for flat discrete-time systems does not yet exist in the literature.
Thus, the aim of the present paper is to close this gap: We investigate
which (lower-order) forward-shifts of a flat output can be used in
principle as a new input, and how to actually introduce such a new
input by a suitable feedback. After discussing the general case, we
consider systems with a flat output that does not depend on future
values of the control input. We show how to systematically construct
a new input such that the sum of the orders of the corresponding forward-shifts
of the components of the flat output is minimal, and that deriving
a linearizing feedback is particularly simple in this case. Furthermore,
we show how such an exact linearization can be used as a basis for
the systematic design of a flatness-based tracking control, and demonstrate
our results by two examples.

The employed mathematical methods are similar as in \cite{GstoettnerKolarSchoberl:2021},
where we have proven that every flat continuous-time system with a
flat output that is independent of time derivatives of the input can
be exactly linearized by a quasi-static feedback of its classical
state. Hence, it is particularly important to emphasize that like
in \cite{GstoettnerKolarSchoberl:2021} we restrict ourselves to feedbacks
of the classical system state, and do not consider feedbacks of generalized
states as e.g. in \cite{DelaleauRudolph:1998} or \cite{Rudolph:2021}.
However, it should be noted that the considered feedbacks are more
general than a usual static or dynamic feedback in the sense that
they may also depend on forward-shifts of the closed-loop input.

The paper is organized as follows: In Section \ref{sec:notation}
and Section \ref{sec:discrete_time_flatness} we introduce some notation
and recapitulate the concept of flatness for discrete-time systems.
In Section \ref{sec:exact_linearization} we investigate the exact
linearization of flat discrete-time systems in general, and derive
further results for systems with a flat output that does not depend
on future values of the control input. The design of a flatness-based
tracking control is discussed in Section \ref{sec:tracking_control},
and in Section \ref{sec:example} the developed theory is finally
applied to the discretized models of a wheeled mobile robot and a
3DOF helicopter.

\section{\label{sec:notation}Notation}

Throughout the paper we make use of some basic differential-geometric
concepts. Let $\mathcal{X}$ be an $n$-dimensional smooth manifold
equipped with local coordinates $x^{i}$, $i=1,\ldots,n$, and $h:\mathcal{X}\rightarrow\mathbb{R}^{m}$
some smooth function. Then we denote by $\partial_{x}h$ the $m\times n$
Jacobian matrix of $h=(h^{1},\ldots,h^{m})$ with respect to $x=(x^{1},\ldots,x^{n})$.
The partial derivative of a single component $h^{j}$ with respect
to a coordinate $x^{i}$ is denoted by $\partial_{x^{i}}h^{j}$. Furthermore,
$\mathrm{d}h^{j}=\partial_{x^{1}}h^{j}\mathrm{d}x^{1}+\ldots+\partial_{x^{n}}h^{j}\mathrm{d}x^{n}$
denotes the differential (exterior derivative) of the function $h^{j}$,
where $\mathrm{d}x^{i}$, $i=1,\ldots,n$ are the differentials corresponding
to the local coordinates. We frequently use $\mathrm{d}h$ as an abbreviation
for the set $\{\mathrm{d}h^{1},\ldots,\mathrm{d}h^{m}\}$, and with
e.g. $\mathrm{span}\{\mathrm{d}h^{1},\ldots,\mathrm{d}h^{m}\}$ we
mean the span over the ring $C^{\infty}(\mathcal{X})$ of smooth functions.
The symbols $\subset$ and $\supset$ are used in the sense that they
also include equality.

To denote forward- and backward-shifts of the system variables, we
use subscripts in brackets. For instance, the $\alpha$-th forward-
or backward-shift of a component $y^{j}$ of a flat output $y$ with
$\alpha\in\mathbb{Z}$ is denoted by $y_{[\alpha]}^{j}$, and $y_{[\alpha]}=(y_{[\alpha]}^{1},\ldots,y_{[\alpha]}^{m})$.
To keep expressions which depend on different numbers of shifts of
different components of a flat output readable, we use multi-indices.
If $A=(a^{1},\ldots,a^{m})$ and $B=(b^{1},\ldots,b^{m})$ are two
multi-indices with $A\leq B$, i.e., $a^{j}\leq b^{j}$ for $j=1,\ldots,m$,
then 
\[
y_{[A]}=(y_{[a^{1}]}^{1},\ldots,y_{[a^{m}]}^{m})
\]
and
\[
y_{[A,B]}=(y_{[a^{1},b^{1}]}^{1},\ldots,y_{[a^{m},b^{m}]}^{m})
\]
with $y_{[a^{j},b^{j}]}^{j}=(y_{[a^{j}]}^{j},\ldots,y_{[b^{j}]}^{j})$.
In the case $a^{j}>b^{j}$ we define $y_{[a^{j},b^{j}]}^{j}$ as empty.
The addition and subtraction of multi-indices is performed componentwise,
and for an integer $c$ we define $A\pm c=(a^{1}\pm c,\ldots,a^{m}\pm c)$.
Furthermore, $\#A=\sum_{j=1}^{m}a^{j}$ denotes the sum over all components
of a multi-index. As an example consider the tuple $y=(y^{1},y^{2})$,
an integer $c=2$, and multi-indices $A=(0,2)$, $B=(1,2)$. We then
have $y_{[c]}=(y_{[2]}^{1},y_{[2]}^{2})$, $y_{[A]}=(y^{1},y_{[2]}^{2})$,
$y_{[A,B]}=(y^{1},y_{[1]}^{1},y_{[2]}^{2})$, and $y_{[A+c]}=(y_{[2]}^{1},y_{[4]}^{2})$
as well as $\#A=2$ and $\#B=3$.

Frequently, it is also convenient to decompose the components of a
flat output $y$ or the input $u$ into several blocks like e.g.
\[
y=(\underbrace{y^{1},\ldots,y^{m_{1}}}_{y_{1}},\underbrace{y^{m_{1}+1},\ldots,y^{m}}_{y_{2}})\,.
\]
Since such blocks are also denoted by a subscript, in this case the
first subscript always refers to the block, and shifts are denoted
by a second subscript in brackets. For instance,
\begin{align*}
y_{[\alpha]} & =(\underbrace{y_{[\alpha]}^{1},\ldots,y_{[\alpha]}^{m_{1}}}_{y_{1,[\alpha]}},\underbrace{y_{[\alpha]}^{m_{1}+1},\ldots,y_{[\alpha]}^{m}}_{y_{2,[\alpha]}})
\end{align*}
with some integer $\alpha\in\mathbb{Z}$, or $y_{1,[A_{1}]}=(y_{1,[a_{1}^{1}]}^{1},\ldots,y_{1,[a_{1}^{m_{1}}]}^{m_{1}})$
with some multi-index $A_{1}=(a_{1}^{1},\ldots,a_{1}^{m_{1}})$.

\section{\label{sec:discrete_time_flatness}Discrete-time systems and flatness}

In this contribution, we consider time-invariant discrete-time nonlinear
systems 
\begin{equation}
x^{i,+}=f^{i}(x,u)\,,\quad i=1,\dots,n\label{eq:sys}
\end{equation}
with $\dim(x)=n$, $\dim(u)=m$ and smooth functions $f^{i}(x,u)$.
In addition, we assume that the system (\ref{eq:sys}) meets the submersivity
condition
\begin{equation}
\mathrm{rank}(\partial_{(x,u)}f)=n\,,\label{eq:submersivity}
\end{equation}
which is common in the discrete-time literature and necessary for
accessibility.

Like in \cite{DiwoldKolarSchoberl:2020}, we call a discrete-time
system (\ref{eq:sys}) flat if there exists a one-to-one correspondence
between its solutions $(x(k),u(k))$ and solutions $y(k)$ of a trivial
system (arbitrary trajectories that need not satisfy any difference
equation) with the same number of inputs. Before we state a more rigorous
definition of discrete-time flatness, let us consider the coupling
of the trajectories $x(k)$ and $u(k)$ by the system equations (\ref{eq:sys}).
By a repeated application of (\ref{eq:sys}), all forward-shifts $x(k+\alpha)$,
$\alpha\geq1$ of the state variables are obviously determined by
$x(k)$ and the input trajectory $u(k+\alpha)$ for $\alpha\geq0$.
A similar argument holds for the backward-direction: Because of the
submersivity condition (\ref{eq:submersivity}), there always exist
$m$ functions $g(x,u)$ such that the $(n+m)\times(n+m)$ Jacobian
matrix
\[
\begin{bmatrix}\partial_{x}f & \partial_{u}f\\
\partial_{x}g & \partial_{u}g
\end{bmatrix}
\]
is regular and the map
\begin{equation}
\begin{array}{ccc}
x^{+} & = & f(x,u)\\
\zeta & = & g(x,u)
\end{array}\label{eq:sys_extension}
\end{equation}
hence locally invertible. By a repeated application of its inverse
\begin{equation}
(x,u)=\psi(x^{+},\zeta)\,,\label{eq:sys_extension_inverse}
\end{equation}
all backward-shifts $x(k-\beta)$, $u(k-\beta)$ of the state- and
input variables for $\beta\geq1$ are determined by $x(k)$ and backward-shifts
$\zeta(k-\beta)$, $\beta\geq1$ of the system variables $\zeta$
defined by (\ref{eq:sys_extension}). Hence, if only a finite time-interval
is considered, the system trajectories $(x(k),u(k))$ can be identified
with points of a manifold $\mathcal{\zeta}_{[-l_{\zeta},-1]}\times\mathcal{X}\times\mathcal{U}_{[0,l_{u}]}$
with coordinates $(\zeta_{[-l_{\zeta}]},\ldots,\zeta_{[-1]},x,u,u_{[1]},\dots,u_{[l_{u}]})$
and suitably chosen integers $l_{\zeta}$, $l_{u}$. If $h\in C^{\infty}(\mathcal{\zeta}_{[-l_{\zeta},-1]}\times\mathcal{X}\times\mathcal{U}_{[0,l_{u}]})$
is a function defined on this manifold, then its future values can
be determined by a repeated application $\delta^{\beta}$ of the forward-shift
operator
\begin{equation}
\delta(h(\dots,\zeta_{[-2]},\zeta_{[-1]},x,u,u_{[1]},\dots))=h(\dots,\zeta_{[-1]},g(x,u),f(x,u),u_{[1]},u_{[2]},\dots)\,.\label{eq:forward_shift_operator}
\end{equation}
Likewise, its past values can be determined by a repeated application
$\delta^{-\beta}$ of the backward-shift operator
\begin{equation}
\delta^{-1}(h(\dots,\zeta_{[-1]},x,u,u_{[1]},u_{[2]},\dots))=h(\dots,\zeta_{[-2]},\psi_{x}(x,\zeta_{[-1]}),\psi_{u}(x,\zeta_{[-1]}),u,u_{[1]},\dots)\,,\label{eq:backward_shift_operator}
\end{equation}
where $\psi_{x}$ and $\psi_{u}$ are the corresponding components
of (\ref{eq:sys_extension_inverse}). Since we work in a finite-dimensional
framework, it is important to emphasize that $\delta$ only yields
the correct forward-shift if the integer $l_{u}$ is chosen large
enough such that the considered function $h$ does not already depend
on $u_{[l_{u}]}$. Likewise, $\delta^{-1}$ only yields the correct
backward-shift if $l_{\zeta}$ is chosen large enough such that the
function $h$ does not already depend on $\zeta_{[-l_{\zeta}]}$.
Thus, throughout this contribution we assume that $l_{u}$ and $l_{\zeta}$
are chosen large enough such that (\ref{eq:forward_shift_operator})
and (\ref{eq:backward_shift_operator}) act as correct forward- and
backward-shifts on all considered functions.

Like the discrete-time static feedback linearization problem, discrete-time
flatness is considered in a suitable neighborhood of an equilibrium
$(x_{0},u_{0})$ of the system (\ref{eq:sys}). On the manifold $\mathcal{\zeta}_{[-l_{\zeta},-1]}\times\mathcal{X}\times\mathcal{U}_{[0,l_{u}]}$,
an equilibrium corresponds to a point with coordinates $(\zeta_{0},\ldots,\zeta_{0},x_{0},u_{0},u_{0},\dots,u_{0})$
with $\zeta_{0}=g(x_{0},u_{0})$ according to (\ref{eq:sys_extension}).
Hence, evaluated at an equilibrium point (or, in other words, at an
equilibrium trajectory), the functions $\delta(h)$ and $\delta^{-1}(h)$
have the same value as the function $h$ itself.
\begin{defn}
\label{def:Flatness}(\cite{DiwoldKolarSchoberl:2020}) The system
(\ref{eq:sys}) is said to be flat around an equilibrium $(x_{0},u_{0})$,
if the $n+m$ coordinate functions $x$ and $u$ can be expressed
locally by an $m$-tuple of functions
\begin{equation}
y^{j}=\varphi^{j}(\zeta_{[-q_{1}]},\dots,\zeta_{[-1]},x,u,\dots,u_{[q_{2}]})\,,\quad j=1,\ldots,m\label{eq:flat_output}
\end{equation}
and their forward-shifts
\begin{equation}
\begin{array}{ccl}
y_{[1]} & = & \delta(\varphi(\zeta_{[-q_{1}]},\dots,\zeta_{[-1]},x,u,\dots,u_{[q_{2}]}))\\
y_{[2]} & = & \delta^{2}(\varphi(\zeta_{[-q_{1}]},\dots,\zeta_{[-1]},x,u,\dots,u_{[q_{2}]}))\\
 & \vdots
\end{array}\label{eq:flat_output_forward_shifts}
\end{equation}
up to some finite order. The $m$-tuple (\ref{eq:flat_output}) is
called a flat output.
\end{defn}

The representation of $x$ and $u$ by a flat output (\ref{eq:flat_output})
is unique and has the form
\begin{equation}
\begin{array}{ccll}
x^{i} & = & F_{x}^{i}(y,\dots,y_{[R-1]})\,,\quad & i=1,\dots,n\\
u^{j} & = & F_{u}^{j}(y,\dots,y_{[R]})\,, & j=1,\dots,m\,.
\end{array}\label{eq:flat_parameterization}
\end{equation}
The multi-index $R=(r^{1},\dots,r^{m})$ consists of the number of
forward-shifts of each component of the flat output (\ref{eq:flat_output})
that are needed to express $x$ and $u$. After substituting (\ref{eq:flat_output_forward_shifts})
into (\ref{eq:flat_parameterization}), the equations are satisfied
identically. Because of Lemma \ref{lem:functional_independence} (see
the appendix), this is equivalent to the condition
\begin{align*}
\mathrm{d}x & \in\mathrm{span}\{\mathrm{d}\varphi,\ldots,\mathrm{d}\delta^{R-1}(\varphi)\}\\
\mathrm{d}u & \in\mathrm{span}\{\mathrm{d}\varphi,\ldots,\mathrm{d}\delta^{R}(\varphi)\}
\end{align*}
which we shall use later. The uniqueness of the map (\ref{eq:flat_parameterization})
is a consequence of the fact that all forward- and backward-shifts
of a flat output are functionally independent, see \cite{DiwoldKolarSchoberl:2020}.
Furthermore, the rows of the Jacobian matrix of the right-hand side
of (\ref{eq:flat_parameterization}) with respect to $y_{[0,R]}$
are linearly independent, i.e., the map (\ref{eq:flat_parameterization})
is a submersion. With a restriction to flat outputs that are independent
of backward-shifts of the system variables, Definition \ref{def:Flatness}
leads to the concept of forward-flatness considered e.g. in \cite{Sira-RamirezAgrawal:2004},
\cite{KaldmaeKotta:2013}, or \cite{KolarKaldmaeSchoberlKottaSchlacher:2016}.
\begin{defn}
\label{def:Forward_Flatness}(\cite{DiwoldKolarSchoberl:2020}) The
system \eqref{eq:sys} is said to be forward-flat, if it meets the
conditions of Definition \ref{def:Flatness} with a flat output of
the form $y^{j}=\varphi^{j}(x,u,\dots,u_{[q_{2}]})$.
\end{defn}

For continuous-time systems, the computation of flat outputs is known
to be a challenging problem. Recent research in this field can be
found e.g. in \cite{NicolauRespondek:2017}, \cite{NicolauRespondek:2019},
or \cite{GstottnerKolarSchoberl:2021b}. For discrete-time systems,
in contrast, we have shown in \cite{KolarSchoberlDiwold:2019} that
every forward-flat system can be decomposed by coordinate transformations
into a smaller-dimensional forward-flat subsystem and an endogenous
dynamic feedback. Because of this property, it is possible to check
the forward-flatness of a system (\ref{eq:sys}) similar to the well-known
static feedback linearization test by computing a certain sequence
of distributions, see \cite{KolarDiwoldSchoberl:2019}. However, even
though ideas for an extension of this approach to the general case
of Definition \ref{def:Flatness} can be found in \cite{Kaldmae:2021},
a computationally efficient test does not yet exist. Hence, within
the present paper, we assume that a flat output is given and do not
address its computation.

\section{\label{sec:exact_linearization}Exact linearization}

In \cite{DiwoldKolarSchoberl:2020}, it has been shown that every
flat discrete-time system (\ref{eq:sys}) can be exactly linearized
by a dynamic feedback which leads to an input-output behaviour of
the form $y_{[r^{j}]}^{j}=v^{j}$, $j=1,\ldots,m$ between the new
input $v=\delta^{R}(\varphi)$ and the considered flat output (\ref{eq:flat_output}).
In the following, we address the question whether also lower-order
forward-shifts $v=\delta^{A}(\varphi)$ of the flat output with a
suitable multi-index $0\leq A\leq R$ can be chosen as new input.
This is particularly interesting for the subsequent design of a tracking
control, since with an input-output behaviour
\[
y_{[a^{j}]}^{j}=v^{j}\,,\quad j=1,\ldots,m
\]
the order of the tracking error dynamics is given by $\#A=\sum_{j=1}^{m}a^{j}$
instead of $\#R=\sum_{j=1}^{m}r^{j}$.
\begin{example}
\textcolor{blue}{\label{exa:introductory_example}}Consider the system
\begin{equation}
\begin{array}{ccl}
x^{1,+} & = & x^{1}+u^{1}\\
x^{2,+} & = & \frac{x^{3}}{u^{1}+1}\\
x^{3,+} & = & u^{2}
\end{array}\label{eq:introductory_example_sys}
\end{equation}
with the flat output
\begin{equation}
\begin{array}{ccccl}
y^{1} & = & \varphi^{1}(x) & = & x^{1}\\
y^{2} & = & \varphi^{2}(x) & = & x^{2}\,.
\end{array}\label{eq:introductory_example_flat_output}
\end{equation}
The corresponding parameterization (\ref{eq:flat_parameterization})
of the system variables by the flat output is given by
\[
\begin{array}{ccl}
x^{1} & = & y^{1}\\
x^{2} & = & y^{2}\\
x^{3} & = & y_{[1]}^{2}\left(1-y^{1}+y_{[1]}^{1}\right)\\
u^{1} & = & y_{[1]}^{1}-y^{1}\\
u^{2} & = & y_{[2]}^{2}\left(1-y_{[1]}^{1}+y_{[2]}^{1}\right)\,,
\end{array}
\]
i.e., there occur forward-shifts of the flat output up to the order
$R=(2,2)$. Thus, as shown in \cite{DiwoldKolarSchoberl:2020}, by
applying a dynamic feedback it is definitely possible to introduce
a new input $v=(\delta^{2}(\varphi^{1}),\delta^{2}(\varphi^{2}))$
and hence achieve an input-output behaviour
\[
\begin{array}{ccl}
y_{[2]}^{1} & = & v^{1}\\
y_{[2]}^{2} & = & v^{2}\,.
\end{array}
\]
However, it is actually also possible to introduce lower-order forward-shifts
of the flat output (\ref{eq:introductory_example_flat_output}) as
a new input: Let us define
\begin{equation}
\begin{array}{ccccl}
v^{1} & = & \delta(\varphi^{1}) & = & x^{1}+u^{1}\\
v^{2} & = & \delta^{2}(\varphi^{2}) & = & \frac{u^{2}}{u_{[1]}^{1}+1}
\end{array}\label{eq:introductory_example_input}
\end{equation}
(i.e., $A=(1,2)$) and complement these equations by the forward-shift
\[
v_{[1]}^{1}=\delta^{2}(\varphi^{1})=x^{1}+u^{1}+u_{[1]}^{1}\,.
\]
If we solve the resulting set of equations for $u^{1}$ and $u^{2}$
as well as $u_{[1]}^{1}$, the original inputs are given by
\begin{equation}
\begin{array}{ccl}
u^{1} & = & v^{1}-x^{1}\\
u^{2} & = & \left(1-v^{1}+v_{[1]}^{1}\right)v^{2}\,.
\end{array}\label{eq:introductory_example_feedback}
\end{equation}
With the feedback\footnote{The feedback (\ref{eq:introductory_example_feedback}) is actually
a discrete-time quasi-static feedback, see e.g. \cite{ArandaKotta:2001}.} (\ref{eq:introductory_example_feedback}), it is possible to introduce
the input (\ref{eq:introductory_example_input}) and achieve an input-output
behaviour
\[
\begin{array}{ccl}
y_{[1]}^{1} & = & v^{1}\\
y_{[2]}^{2} & = & v^{2}\,.
\end{array}
\]
However, as already mentioned in the introduction, this requires a
feedback which also depends on forward-shifts of the new input.
\end{example}

In this introductory example, we have chosen the new input $v=(\delta(\varphi^{1}),\delta^{2}(\varphi^{2}))$
without a prior theoretical justification. The criterion for the feasibility
of an $m$-tuple of forward-shifts $\delta^{A}(\varphi)$ of a flat
output (\ref{eq:flat_output}) as a new input $v$ is the possibility
to realize arbitrary trajectories $v(k+\alpha)$, $\alpha\geq0$ independently
of the previous trajectory of the system. More precisely, like for
the original input $u$, at every time step $k$ the system (\ref{eq:sys})
must permit arbitrary trajectories $v(k+\alpha)$, $\alpha\geq0$
independently of its current state $x(k)$ and past values $\zeta(k-\beta)$,
$\beta\geq1$. That is, for every possible state $x(k)$ and past
values $\zeta(k-\beta)$, $\beta\geq1$ there must exist a trajectory
$u(k+\alpha)$, $\alpha\geq0$ of the original control input such
that the desired trajectory $v(k+\alpha)$, $\alpha\geq0$ can be
realized. The practical importance of this criterion is obvious, since
otherwise there would be no guarantee that a trajectory $v(k+\alpha)$,
$\alpha\geq0$ requested e.g. by a controller for the exactly linearized
system can actually be achieved. Similar considerations for flat continuous-time
systems can be found in \cite{GstoettnerKolarSchoberl:2021}, where
it has to be ensured that arbitrary trajectories $v(t)$ can be realized.
\begin{thm}
\label{thm:Conditions_Input}Let $A=(a^{1},\ldots,a^{m})$ denote
a multi-index with $a^{j}\geq0$, $j=1,\ldots,m$. The system (\ref{eq:sys})
permits arbitrary trajectories $v(k+\alpha)$, $\alpha\geq0$ for
the forward-shifts $v=\delta^{A}(\varphi)$ of a flat output (\ref{eq:flat_output})
regardless of its current state $x(k)$ and past values $\zeta(k-\beta)$,
$\beta\geq1$ if and only if the differentials
\begin{equation}
\mathrm{d}\zeta_{[-q_{1}+A]},\ldots,\mathrm{d}\zeta_{[-1]},\mathrm{d}x,\mathrm{d}\delta^{A}(\varphi),\ldots,\mathrm{d}\delta^{R-1}(\varphi)\label{eq:Condition_Input}
\end{equation}
are linearly independent.
\end{thm}

\begin{proof}
The system permits arbitrary trajectories $v(k+\alpha)$, $\alpha\geq0$
if and only if there does not exist any nontrivial relation of the
form
\begin{equation}
\psi(\ldots,\zeta(k-2),\zeta(k-1),x(k),v(k),v(k+1),\ldots)=0\,.\label{eq:nontrivial_relation}
\end{equation}
Otherwise, (\ref{eq:nontrivial_relation}) could be solved by the
implicit function theorem for at least one component of some $v(k+\alpha)$,
$\alpha\geq0$, which would thus be uniquely determined by the other
quantities appearing in (\ref{eq:nontrivial_relation}). In our differential-geometric
framework with the manifold $\mathcal{\zeta}_{[-l_{\zeta},-1]}\times\mathcal{X}\times\mathcal{U}_{[0,l_{u}]}$,
this corresponds to the non-existence of any nontrivial relation
\[
\psi(\ldots,\zeta_{[-2]},\zeta_{[-1]},x,\delta^{A}(\varphi),\delta^{A+1}(\varphi),\ldots)=0\,.
\]
Because of Lemma \ref{lem:functional_independence}, this condition
is equivalent to the linear independence of the differentials
\begin{equation}
\ldots,\mathrm{d}\zeta_{[-2]},\mathrm{d}\zeta_{[-1]},\mathrm{d}x,\mathrm{d}\delta^{A}(\varphi),\mathrm{d}\delta^{A+1}(\varphi),\ldots\,.\label{eq:infinite_set_differentials}
\end{equation}
However, it is not necessary to check the linear independence of all
these differentials. Since the flat output (\ref{eq:flat_output})
is independent of the variables $\ldots,\zeta_{[-q_{1}-2]},\zeta_{[-q_{1}-1]}$,
its forward-shift $\delta^{A}(\varphi)$ is independent of $\ldots,\zeta_{[-q_{1}+A-2]},\zeta_{[-q_{1}+A-1]}$.
Thus, we do not need to consider the corresponding differentials.
Furthermore, since the differentials of a flat output and all its
forward- and backward-shifts are linearly independent and the fact
that
\[
\mathrm{d}x\in\mathrm{span}\{\mathrm{d}\varphi,\ldots,\mathrm{d}\delta^{R-1}(\varphi)\}
\]
as well as\footnote{Note that because of (\ref{eq:sys_extension}) the quantities $\zeta$
are functions of $x$ and $u$. Hence, their parameterization by the
flat output (\ref{eq:flat_output}) can be obtained immediately from
(\ref{eq:flat_parameterization}).}
\[
\begin{array}{ccl}
\mathrm{d}\zeta_{[-1]} & \in & \mathrm{span}\{\mathrm{d}\delta^{-1}(\varphi),\mathrm{d}\varphi,\ldots,\mathrm{d}\delta^{R-1}(\varphi)\}\\
\mathrm{d}\zeta_{[-2]} & \in & \mathrm{span}\{\mathrm{d}\delta^{-2}(\varphi),\mathrm{d}\delta^{-1}(\varphi),\ldots,\mathrm{d}\delta^{R-2}(\varphi)\}\,,\\
 & \vdots
\end{array}
\]
there is also no need to consider the differentials $\mathrm{d}\delta^{R}(\varphi),\mathrm{d}\delta^{R+1}(\varphi),\ldots$.
Consequently, the linear independence of the differentials (\ref{eq:Condition_Input})
implies the linear independence of the differentials (\ref{eq:infinite_set_differentials}),
which completes the proof.
\end{proof}
An immediate consequence of Theorem \ref{thm:Conditions_Input} is
that the choice of an input $v=\delta^{A}(\varphi)$ with $A\geq R$
is always possible.
\begin{example}
Consider again the system (\ref{eq:introductory_example_sys}) with
the flat output (\ref{eq:introductory_example_flat_output}) of Example
\ref{exa:introductory_example}. For the chosen new input (\ref{eq:introductory_example_input})
with $A=(1,2)$ and $R=(2,2)$, the differentials (\ref{eq:Condition_Input})
of Theorem \ref{thm:Conditions_Input} are given by
\[
\begin{array}{rcc}
\mathrm{d}x^{1}\\
\mathrm{d}x^{2}\\
\mathrm{d}x^{3}\\
\mathrm{d}\delta(\varphi^{1}) & = & \mathrm{d}x^{1}+\mathrm{d}u^{1}
\end{array}
\]
and obviously linearly independent. Thus, the system permits indeed
arbitrary trajectories $v(k+\alpha)$, $\alpha\geq0$ for the new
input (\ref{eq:introductory_example_input}) independently of its
current state $x(k)$ and past values of the system trajectory. The
latter play no role in this case, since the flat output (\ref{eq:introductory_example_flat_output})
is a forward-flat output. If, however, we would try to use the flat
output (\ref{eq:introductory_example_flat_output}) itself as a new
input
\[
\begin{array}{ccccl}
v^{1} & = & \varphi^{1} & = & x^{1}\\
v^{2} & = & \varphi^{2} & = & x^{2}\,,
\end{array}
\]
then it is obvious that the possible trajectories $v(k+\alpha)$,
$\alpha\geq0$ are restricted by the current state $x(k)$. Accordingly,
it can be observed that the differentials
\[
\begin{array}{rcl}
\mathrm{d}x^{1}\\
\mathrm{d}x^{2}\\
\mathrm{d}x^{3}\\
\mathrm{d}\varphi^{1} & = & \mathrm{d}x^{1}\\
\mathrm{d}\varphi^{2} & = & \mathrm{d}x^{2}\\
\mathrm{d}\delta(\varphi^{1}) & = & \mathrm{d}x^{1}+\mathrm{d}u^{1}\\
\mathrm{d}\delta(\varphi^{2}) & = & \frac{1}{u^{1}+1}\mathrm{d}x^{3}-\frac{x^{3}}{(u^{1}+1)^{2}}\mathrm{d}u^{1}
\end{array}
\]
of condition (\ref{eq:Condition_Input}) with $A=(0,0)$ are not linearly
independent.
\end{example}

\subsection{Construction of the linearizing feedback in the general case}

Theorem \ref{thm:Conditions_Input} ensures that every trajectory
$v(k+\alpha)$, $\alpha\geq0$ of the new input $v=\delta^{A}(\varphi)$
can be realized independently of the previous trajectory of the system
up to the time instant $k$ by applying a suitable trajectory of the
control input $u(k+\alpha)$, $\alpha\geq0$. In the following, we
show how the required trajectory $u(k+\alpha)$ can be determined
by a suitable state feedback. In other words, we derive a feedback
which actually introduces $v=\delta^{A}(\varphi)$ as new input. For
simplicity we assume $A\leq R$, since the choice $A=R$ is possible
anyway. The basic idea for the construction of the linearizing feedback
is similar as in the continuous-time case in \cite{GstoettnerKolarSchoberl:2021}.
However, due to the different transformation laws of continuous-time
and discrete-time systems, the proof is adapted accordingly.

Because of the linear independence of the differentials (\ref{eq:Condition_Input})
and $\mathrm{d}x\in\mathrm{span}\{\mathrm{d}\varphi_{[0,R-1]}\}$,
there exists a selection $\mathrm{d}\varphi_{c}$ of $\#A-n$ differentials
from the set $\mathrm{d}\varphi_{[0,A-1]}$ such that
\[
\mathrm{span}\{\mathrm{d}\varphi_{[0,R]}\}=\mathrm{span}\{\mathrm{d}\varphi_{[0,A-1]},\mathrm{d}\varphi_{[A,R]}\}=\mathrm{span}\{\mathrm{d}x,\mathrm{d}\varphi_{c},\mathrm{d}\varphi_{[A,R]}\}\,.
\]
As a consequence of Lemma \ref{lem:functional_independence}, there
exists a diffeomorphism $\Psi:\mathbb{R}^{\#R+m}\rightarrow\mathbb{R}^{\#R+m}$
such that locally
\begin{equation}
\varphi_{[0,R]}=\Psi(x,\varphi_{c},\varphi_{[A,R]})\label{eq:Diff_Psi}
\end{equation}
holds identically. Its inverse is given by
\[
\begin{array}{rcl}
x & = & F_{x}(\varphi_{[0,R-1]})\\
\varphi_{c} & = & \varphi_{c}\\
\varphi_{[A,R]} & = & \varphi_{[A,R]}\,,
\end{array}
\]
where $F_{x}$ is the parameterization of the state variables according
to (\ref{eq:flat_parameterization}). Based on the diffeomorphism
(\ref{eq:Diff_Psi}), the quantities $v=\delta^{A}(\varphi)$ can
be introduced as new input of the system (\ref{eq:sys}) by a dynamic
feedback with the controller state $z=\varphi_{c}$. Since the functions
$\varphi_{c}$ belong to the set $\varphi_{[0,A-1]}$, their forward-shifts
$\varphi_{c,[1]}=\delta(\varphi_{c})$ belong to the set $\varphi_{[1,A]}\subset\varphi_{[0,R]}$
and can hence be expressed as functions of $x$, $\varphi_{c}$, and
$\varphi_{[A,R]}$. The corresponding components of (\ref{eq:Diff_Psi})
are denoted in the following as $\varphi_{c,[1]}=\psi_{c,[1]}(x,\varphi_{c},\varphi_{[A,R]})$.
\begin{thm}
\label{thm:exact_linearization}Consider a system (\ref{eq:sys})
with a flat output (\ref{eq:flat_output}) and a multi-index $A\leq R$
which satisfies the condition of Theorem \ref{thm:Conditions_Input}.
With a feedback
\begin{equation}
\begin{array}{ccl}
z^{+} & = & \psi_{c,[1]}(x,z,v_{[0,R-A]})\\
u & = & F_{u}\circ\Psi(x,z,v_{[0,R-A]})
\end{array}\label{eq:Linearizing_Feedback}
\end{equation}
with $\dim(z)=\#A-n$, the closed-loop system
\begin{equation}
\begin{array}{rcl}
x^{+} & = & f(x,F_{u}\circ\Psi(x,z,v_{[0,R-A]}))\\
z^{+} & = & \psi_{c,[1]}(x,z,v_{[0,R-A]})
\end{array}\label{eq:Closed_Loop_System}
\end{equation}
has the input-output behaviour $y_{[A]}=v$.
\end{thm}

\begin{proof}
First, let us extend the feedback (\ref{eq:Linearizing_Feedback})
by the trivial equations $v_{[0,R-A-1]}^{+}=v_{[1,R-A]}$, such that
the extended closed-loop system
\begin{equation}
\begin{array}{rcl}
x^{+} & = & f(x,F_{u}\circ\Psi(x,z,v_{[0,R-A]}))\\
z^{+} & = & \psi_{c,[1]}(x,z,v_{[0,R-A]})\\
v_{[0,R-A-1]}^{+} & = & v_{[1,R-A]}
\end{array}\label{eq:Extended_Closed_Loop_System}
\end{equation}
has the form of a classical state representation with the input $v_{[R-A]}$.
In the following, we show that with the transformation
\begin{equation}
y_{[0,R]}=\Psi(x,z,v_{[0,R-A]})\label{eq:Transformation_Brunovsky}
\end{equation}
derived from (\ref{eq:Diff_Psi}) the system is equivalent to the
discrete-time Brunovsky normal form
\begin{equation}
\begin{array}{rclcrcl}
y^{1,+} & = & y_{[1]}^{1} & \cdots & y^{m,+} & = & y_{[1]}^{m}\\
 & \vdots &  &  &  & \vdots\\
y_{[r^{1}-1]}^{1,+} & = & y_{[r^{1}]}^{1} & \cdots & y_{[r^{m}-1]}^{m,+} & = & y_{[r^{m}]}^{m}\,.
\end{array}\label{eq:Brunovsky_Normal_Form}
\end{equation}
From the inverse
\begin{equation}
\begin{array}{rcl}
x & = & F_{x}(y_{[0,R-1]})\\
z & = & y_{c}\\
v_{[0,R-A-1]} & = & y_{[A,R-1]}\\
v_{[R-A]} & = & y_{[R]}
\end{array}\label{eq:Transformation_Brunovsky_Inverse}
\end{equation}
of (\ref{eq:Transformation_Brunovsky}) and $y_{c}\subset y_{[0,A-1]}\subset y_{[0,R-1]}$,
it is clear that the transformation is actually a state transformation
for the extended closed-loop system (\ref{eq:Extended_Closed_Loop_System}).
The input is only renamed according to $v_{[R-A]}=y_{[R]}$. Because
of the transformation law for discrete-time systems, applying the
transformation (\ref{eq:Transformation_Brunovsky_Inverse}) to the
Brunovsky normal form (\ref{eq:Brunovsky_Normal_Form}) yields
\[
\begin{array}{rcl}
x^{+} & = & F_{x}(y_{[1,R]})\circ\Psi(x,z,v_{[0,R-A]})\\
z^{+} & = & y_{c,[1]}\circ\Psi(x,z,v_{[0,R-A]})\\
v_{[0,R-A-1]}^{+} & = & y_{[A+1,R]}\circ\Psi(x,z,v_{[0,R-A]})\,.
\end{array}
\]
Using the identity\footnote{If the parameterization (\ref{eq:flat_parameterization}) of the system
variables is substituted into (\ref{eq:sys}), then the equations
are satisfied identically.} $F_{x}(y_{[1,R]})=f(F_{x}(y_{[0,R-1]}),F_{u}(y_{[0,R]}))$ as well
as $F_{x}\circ\Psi(x,z,v_{[0,R-A]})=x$ and $y_{[A+1,R]}\circ\Psi(x,z,v_{[0,R-A]})=v_{[1,R-A]}$,
the system representation (\ref{eq:Extended_Closed_Loop_System})
follows. Thus, the extended closed-loop system (\ref{eq:Extended_Closed_Loop_System})
is equivalent to the Brunovsky normal form (\ref{eq:Brunovsky_Normal_Form})
via a state transformation and a renaming of the input. Consequently,
it has the linear input-output behaviour $y_{[R]}=v_{[R-A]}$. Since
the closed-loop system (\ref{eq:Closed_Loop_System}) has the input
$v$ instead of $v_{[R-A]}$, it has the input-output behaviour $y_{[A]}=v$.
\end{proof}
In contrast to a classical static or dynamic feedback,
the feedback (\ref{eq:Linearizing_Feedback}) depends besides the
new input $v$ also on its forward-shifts up to the order $R-A$.
This is similar to the continuous-time case, where feedbacks which
depend also on time derivatives of the new input have been used successfully
for the exact linearization of flat systems since the nineties of
the last century, see e.g. \cite{DelaleauRudolph:1998} or \cite{Rudolph:2021}.
For a practical application, this means that at every time step not
only the value of $v$ itself needs to be specified but also its future
values. Hence, if a control law for the exactly linearized system
$y_{[A]}=v$ is designed, also expressions for the forward-shifts
of $v$ occuring in (\ref{eq:Linearizing_Feedback}) have to be derived.
If the system $y_{[A]}=v$ shall be controlled by a pure feedforward
control, this is of course straightforward as long as the desired
reference trajectory $y^{d}$ is known a sufficient number of time
steps ahead (note that with such a feedforward control already a dead-beat
behaviour can be achieved). In Section \ref{sec:tracking_control},
it is shown how the required forward-shifts of $v$ can be determined
also for a more general type of tracking control.
\begin{rem}
For $\#A=n$, the controller state $z$ is empty and the feedback
(\ref{eq:Linearizing_Feedback}) degenerates to a feedback of the
form $u=F_{u}\circ\Psi(x,v_{[0,R-A]})$. Since there is no controller
state but the feedback depends on forward-shifts of the closed-loop
input $v$, such a feedback is a discrete-time quasi-static feedback
as it is defined in \cite{ArandaKotta:2001}. For continuous-time
systems, a quasi-static feedback depends on time derivatives instead
of forward-shifts of the closed-loop input, see \cite{DelaleauRudolph:1998}
or \cite{Rudolph:2021}.
\end{rem}

\subsection{\label{subsec:quasistatic_linearization_xu_flat}Flat outputs that
are independent of future values of the input}

In the remainder of the paper, we consider flat outputs of the form
\begin{equation}
y^{j}=\varphi^{j}(\zeta_{[-q_{1}]},\dots,\zeta_{[-1]},x,u)\,,\quad j=1,\ldots,m\label{eq:flat_output_xu}
\end{equation}
that are independent of future values of the input $u$. With this
restriction, it is possible to derive further results in a similar
way as in \cite{GstoettnerKolarSchoberl:2021} for flat outputs of
continuous-time systems which are independent of time derivatives
of $u$. In the following, we show how to systematically construct
a ``minimal'' multi-index $\kappa=(\kappa^{1},\ldots,\kappa^{m})$
such that with $A=\kappa$ the condition of Theorem \ref{thm:Conditions_Input}
is met and $\#\kappa\leq\#A$ for all other feasible multi-indices
$A$. The basic idea is to replace the coordinates $u,u_{[1]},u_{[2]},\ldots$
of the manifold $\mathcal{\zeta}_{[-l_{\zeta},-1]}\times\mathcal{X}\times\mathcal{U}_{[0,l_{u}]}$
step by step by forward-shifts $v,v_{[1]},v_{[2]},\ldots$ of the
flat output (\ref{eq:flat_output_xu}) with $v_{[\alpha]}=\delta^{\kappa+\alpha}(\varphi)$,
$\alpha\geq0$, such that finally we have coordinates $(\ldots\zeta_{[-2]},\zeta_{[-1]},x,v,v_{[1]},v_{[2]},\ldots)$.\footnote{Since the manifold $\mathcal{\zeta}_{[-l_{\zeta},-1]}\times\mathcal{X}\times\mathcal{U}_{[0,l_{u}]}$
is finite-dimensional, in fact some of the higher-order forward-shifts
of $u$ cannot be replaced by forward-shifts of the flat output and
must be kept as coordinates (unless the system (\ref{eq:sys}) is
static feedback linearizable and (\ref{eq:flat_output_xu}) a linearizing
output).} For this purpose, we forward-shift every component of the flat output
(\ref{eq:flat_output_xu}) until it depends explicitly on the input
$u$, and introduce as many of these functions as possible as new
coordinates. Subsequently, the other components of the flat output
are further shifted until they depend explicitly on the remaining
components of $u$, and again as many of these functions as possible
are introduced as new coordinates. Continuing this procedure until
all $m$ components of the original input $u$ have been replaced
by forward-shifts of the flat output (\ref{eq:flat_output_xu}) yields
a minimal multi-index $\kappa=(\kappa^{1},\ldots,\kappa^{m})$ such
that with $v=\delta^{\kappa}(\varphi)$ the condition of Theorem \ref{thm:Conditions_Input}
is satisfied. In the following, we explain the procedure in detail.

In the first step, determine the multi-index $K_{1}=(k_{1}^{1},\ldots,k_{1}^{m})$
such that
\begin{align*}
\delta^{k_{1}^{j}-1}(\varphi^{j}) & =\varphi_{[k_{1}^{j}-1]}^{j}(\zeta_{[-q_{1},-1]},x)\\
\delta^{k_{1}^{j}}(\varphi^{j}) & =\varphi_{[k_{1}^{j}]}^{j}(\zeta_{[-q_{1},-1]},x,u)
\end{align*}
and define $m_{1}=\mathrm{rank}(\partial_{u}\varphi_{[K_{1}]})$.
Then reorder the components of the flat output (\ref{eq:flat_output_xu})
and the input $u$ such that $\mathrm{rank}(\partial_{u_{1}}\varphi_{1,[\kappa_{1}]})=m_{1}$,
where $\varphi_{1}=(\varphi^{1},\ldots,\varphi^{m_{1}})$, $u_{1}=(u^{1},\ldots,u^{m_{1}})$,
and $\kappa_{1}=(k_{1}^{1},\ldots,k_{1}^{m_{1}})$ consist of the
first $m_{1}$ components of $\varphi$, $u$, and $K_{1}$, respectively.
Now apply the coordinate transformation
\begin{align}
\begin{aligned}v_{1} & =\varphi_{1,[\kappa_{1}]}(\zeta_{[-q_{1},-1]},x,u)\\
u_{rest_{1}} & =(u^{m_{1}+1},\ldots,u^{m})\\[1ex]
v_{1,[1]} & =\varphi_{1,[\kappa_{1}+1]}(\zeta_{[-q_{1},-1]},x,u,u_{[1]})\\
u_{rest_{1},[1]} & =(u_{[1]}^{m_{1}+1},\ldots,u_{[1]}^{m})\\[1ex]
v_{1,[2]} & =\varphi_{1,[\kappa_{1}+2]}(\zeta_{[-q_{1},-1]},x,u,u_{[1]},u_{[2]})\\
u_{rest_{1},[2]} & =(u_{[2]}^{m_{1}+1},\ldots,u_{[2]}^{m})\\
 & \:\:\vdots
\end{aligned}
\label{eq:trans_1}
\end{align}
which replaces the inputs $u_{1}$ and their forward-shifts by $v_{1}$
and its forward-shifts. The remaining inputs $u_{rest_{1}}=(u^{m_{1}+1},\ldots,u^{m})$
and their forward-shifts are left unchanged.
\begin{rem}
The coordinate transformation (\ref{eq:trans_1}) is indeed regular,
since in a sufficiently small neighborhood of an equilibrium point
the condition $\mathrm{rank}(\partial_{u_{1}}\varphi_{1,[\kappa_{1}]})=m_{1}$
implies $\mathrm{rank}(\partial_{u_{1,[\alpha]}}\varphi_{1,[\kappa_{1}+\alpha]})=m_{1}$
for $\alpha\geq1$. In the new coordinates, the forward-shift operator
(\ref{eq:forward_shift_operator}) has the form
\begin{multline*}
\delta(h(\dots,\zeta_{[-2]},\zeta_{[-1]},x,v_{1},u_{rest_{1}},v_{1,[1]},u_{rest_{1},[1]},\dots))\\
=h(\dots,\zeta_{[-1]},g(x,\hat{\Phi}),f(x,\hat{\Phi}),v_{1,[1]},u_{rest_{1},[1]},v_{1,[2]},u_{rest_{1},[2]},\dots)
\end{multline*}
with $\hat{\Phi}$ denoting the inverse of the transformation (\ref{eq:trans_1}).
\end{rem}

After the coordinate transformation (\ref{eq:trans_1}) we have
\begin{align*}
y_{1,[0,\kappa_{1}-1]} & =\varphi_{1,[0,\kappa_{1}-1]}(\zeta_{[-q_{1},-1]},x)\\
y_{1,[\kappa_{1}]} & =v_{1}\\
y_{rest_{1},[0,K_{rest_{1}}-1]} & =\varphi_{rest_{1},[0,K_{rest_{1}}-1]}(\zeta_{[-q_{1},-1]},x)\\
y_{rest_{1},[K_{rest_{1}}]} & =\varphi_{rest_{1},[K_{rest_{1}}]}(\zeta_{[-q_{1},-1]},x,v_{1})
\end{align*}
with $y_{1}=(y^{1},\ldots,y^{m_{1}})$, $y_{rest_{1}}=(y^{m_{1}+1},\ldots,y^{m})$,
$K_{rest_{1}}=(k_{1}^{m_{1}+1},\ldots,k_{1}^{m})$, and $\varphi_{rest_{1},[\alpha]}=(\varphi_{[\alpha]}^{m_{1}+1},\ldots,\varphi_{[\alpha]}^{m})\circ\hat{\Phi}$.
The functions $\varphi_{rest_{1},[K_{rest_{1}}]}$ are independent
of $u_{rest_{1}}$, since otherwise $\mathrm{rank}(\partial_{u}\varphi_{[K_{1}]})$
would have been larger than $m_{1}$.

In the second step, determine the multi-index $K_{2}=(k_{2}^{1},\ldots,k_{2}^{m-m_{1}})$
such that
\begin{align*}
\delta^{k_{2}^{j}-1}(\varphi_{rest_{1}}^{j}) & =\varphi_{rest_{1},[k_{2}^{j}-1]}^{j}(x,v_{1},v_{1,[1]},\ldots)\\
\delta^{k_{2}^{j}}(\varphi_{rest_{1}}^{j}) & =\varphi_{rest_{1},[k_{2}^{j}]}^{j}(x,v_{1},v_{1,[1]},\ldots,u_{rest_{1}})
\end{align*}
and let $m_{2}=\mathrm{rank}(\partial_{u_{rest_{1}}}\varphi_{rest_{1},[K_{2}]})$.
Then reorder the components of the flat output belonging to $y_{rest_{1}}$
and the components of the input belonging to $u_{rest_{1}}$ such
that $\mathrm{rank}(\partial_{u_{2}}\varphi_{2,[\kappa_{2}]})=m_{2}$,
where $\varphi_{2}=(\varphi_{rest_{1}}^{1},\ldots,\varphi_{rest_{1}}^{m_{2}})$,
$u_{2}=(u_{rest_{1}}^{1},\ldots,u_{rest_{1}}^{m_{2}})$, and $\kappa_{2}=(k_{2}^{1},\ldots,k_{2}^{m_{2}})$
consist of the first $m_{2}$ components of $\varphi_{rest_{1}}$,
$u_{rest_{1}}$, and $K_{2}$, respectively. Now apply the coordinate
transformation\footnote{The functions $\varphi_{2,[\kappa_{2}]}$ depend on forward-shifts
of $v_{1}$, and since we work on a finite-dimensional manifold $\mathcal{\zeta}_{[-l_{\zeta},-1]}\times\mathcal{X}\times\mathcal{U}_{[0,l_{u}]}$
these forward-shifts are only available up to the order $l_{u}$.
Thus, some higher-order forward-shifts of $u_{2}$ must be kept as
coordinates on $\mathcal{\zeta}_{[-l_{\zeta},-1]}\times\mathcal{X}\times\mathcal{U}_{[0,l_{u}]}$
and cannot be replaced by forward-shifts of $v_{2}$. However, as
long as $l_{u}$ is chosen sufficiently large, this does not affect
our considerations.}
\begin{align}
\begin{aligned}v_{2} & =\varphi_{2,[\kappa_{2}]}(x,v_{1},v_{1,[1]},\ldots,u_{rest_{1}})\\
u_{rest_{2}} & =(u^{m_{1}+m_{2}+1},\ldots,u^{m})\\[1ex]
v_{2,[1]} & =\varphi_{2,[\kappa_{2}+1]}(x,v_{1},v_{1,[1]},\ldots,u_{rest_{1}},u_{rest_{1},[1]})\\
u_{rest_{2},[1]} & =(u_{[1]}^{m_{1}+m_{2}+1},\ldots,u_{[1]}^{m})\\[1ex]
v_{2,[2]} & =\varphi_{2,[\kappa_{2}+2]}(x,v_{1},v_{1,[1]},\ldots,u_{rest_{1}},u_{rest_{1},[1]},u_{rest_{1},[2]})\\
u_{rest_{2},[2]} & =(u_{[2]}^{m_{1}+m_{2}+1},\ldots,u_{[2]}^{m})\\
 & \:\:\vdots
\end{aligned}
\label{eq:trans_2}
\end{align}
which replaces the inputs $u_{2}=(u^{m_{1}+1},\ldots,u^{m_{1}+m_{2}})$
and their forward-shifts by $v_{2}$ and its forward-shifts. The remaining
inputs $u_{rest_{2}}=(u^{m_{1}+m_{2}+1},\ldots,u^{m})$ and their
forward-shifts are left unchanged. After the coordinate transformation
(\ref{eq:trans_2}) we have
\begin{align*}
y_{1,[0,\kappa_{1}-1]} & =\varphi_{1,[0,\kappa_{1}-1]}(\zeta_{[-q_{1},-1]},x)\\
y_{1,[\kappa_{1}]} & =v_{1}\\
y_{2,[0,\kappa_{2}-1]} & =\varphi_{2,[0,\kappa_{2}-1]}(\zeta_{[-q_{1},-1]},x,v_{1},v_{1,[1]},\ldots)\\
y_{2,[\kappa_{2}]} & =v_{2}\\
y_{rest_{2},[0,K_{rest_{2}}-1]} & =\varphi_{rest_{2},[0,K_{rest_{2}}-1]}(\zeta_{[-q_{1},-1]},x,v_{1},v_{1,[1]},\ldots)\\
y_{rest_{2},[K_{rest_{2}}]} & =\varphi_{rest_{2},[K_{rest_{2}}]}(\zeta_{[-q_{1},-1]},x,v_{1},v_{1,[1]},\ldots,v_{2})\,,
\end{align*}
where $y_{2}=(y^{m_{1}+1},\ldots,y^{m_{1}+m_{2}})$, $y_{rest_{2}}=(y^{m_{1}+m_{2}+1},\ldots,y^{m})$,
$K_{rest_{2}}=(k_{2}^{m_{2}+1},\ldots,k_{2}^{m-m_{1}})$, and $\varphi_{rest_{2},[\alpha]}=(\varphi_{rest_{1},[\alpha]}^{m_{2}+1},\ldots,\varphi_{rest_{1},[\alpha]}^{m-m_{1}})\circ\hat{\Phi}$
with the inverse $\hat{\Phi}$ of the transformation (\ref{eq:trans_2}).
The functions $\varphi_{rest_{2},[K_{rest_{2}}]}$ are again independent
of $u_{rest_{2}}$, since otherwise $\mathrm{rank}(\partial_{u_{rest_{1}}}\varphi_{rest_{1},[K_{2}]})$
would have been larger than $m_{2}$.

This procedure is now continued until in some step $s$ we obtain
a multi-index $K_{s}=(k_{s}^{1},\ldots,k_{s}^{m-m_{1}-\ldots-m_{s-1}})$
with
\begin{align*}
\delta^{k_{s}^{j}-1}(\varphi_{rest_{s-1}}^{j}) & =\varphi_{rest_{s-1},[k_{s}^{j}-1]}^{j}(x,v_{1},v_{1,[1]},\ldots,v_{s-1},v_{s-1,[1]},\ldots)\\
\delta^{k_{s}^{j}}(\varphi_{rest_{s-1}}^{j}) & =\varphi_{rest_{s-1},[k_{s}^{j}]}^{j}(x,v_{1},v_{1,[1]},\ldots,v_{s-1},v_{s-1,[1]},\ldots,u_{rest_{s-1}})
\end{align*}
such that $\mathrm{rank}(\partial_{u_{rest_{s-1}}}\varphi_{rest_{s-1},[K_{s}]})=\dim(u_{rest_{s-1}})$.
Thus, with $\varphi_{s}=\varphi_{rest_{s-1}}$ and $\kappa_{s}=K_{s}$
we can apply the coordinate transformation
\begin{align*}
v_{s} & =\varphi_{s,[\kappa_{s}]}(x,v_{1},v_{1,[1]},\ldots,v_{s-1},v_{s-1,[1]},\ldots,u_{rest_{s-1}})\,,
\end{align*}
which replaces the remaining inputs $u_{rest_{s-1}}$ by $v_{s}$.
With the constructed coordinates, the flat output and its forward-shifts
up to the orders $\kappa_{i}$ are finally given by
\begin{equation}
\begin{array}{rcl}
y_{1,[0,\kappa_{1}-1]} & = & \varphi_{1,[0,\kappa_{1}-1]}(\zeta_{[-q_{1},-1]},x)\\
y_{1,[\kappa_{1}]} & = & v_{1}\\
y_{2,[0,\kappa_{2}-1]} & = & \varphi_{2,[0,\kappa_{2}-1]}(\zeta_{[-q_{1},-1]},x,v_{1},v_{1,[1]},\ldots)\\
y_{2,[\kappa_{2}]} & = & v_{2}\\
 & \vdots\\
y_{s-1,[0,\kappa_{s-1}-1]} & = & \varphi_{s-1,[0,\kappa_{s-1}-1]}(\zeta_{[-q_{1},-1]},x,v_{1},v_{1,[1]},\ldots,v_{s-2},v_{s-2,[1]},\ldots)\\
y_{s-1,[\kappa_{s-1}]} & = & v_{s-1}\\
y_{s,[0,\kappa_{s}-1]} & = & \varphi_{s,[0,\kappa_{s}-1]}(\zeta_{[-q_{1},-1]},x,v_{1},v_{1,[1]},\ldots,v_{s-1},v_{s-1,[1]},\ldots)\\
y_{s,[\kappa_{s}]} & = & v_{s}
\end{array}\label{eq:new _coordinates}
\end{equation}
with $\dim(y_{i})=m_{i}$ and $\kappa_{i}=(\kappa_{i}^{1},\ldots,\kappa_{i}^{m_{i}})$.
\begin{thm}
For every flat output (\ref{eq:flat_output_xu}) of the system (\ref{eq:sys}),
the above procedure terminates after $s\leq m$ steps. The multi-index
$\kappa=(\kappa_{1},\ldots,\kappa_{s})$ formed by the constructed
multi-indices $\kappa_{i}=(\kappa_{i}^{1},\ldots,\kappa_{i}^{m_{i}})$
has the following properties:
\begin{enumerate}
\item[i)] $\kappa\leq R$
\item[ii)] $\#\kappa\geq n$, and $\#\kappa=n$ if and only if the flat output
(\ref{eq:flat_output_xu}) is independent of the variables $\zeta_{[-q_{1}]},\dots,\zeta_{[-1]}$.
\item[iii)] With $A=\kappa$ the condition of Theorem \ref{thm:Conditions_Input}
is met, and $\#\kappa\leq\#A$ for all other multi-indices $A$ that
satisfy this condition.
\end{enumerate}
Furthermore, (\ref{eq:new _coordinates}) is actually of the form
\begin{equation}
\begin{array}{rcl}
y_{1,[0,\kappa_{1}-1]} & = & \varphi_{1,[0,\kappa_{1}-1]}(\zeta_{[-q_{1},-1]},x)\\
y_{1,[\kappa_{1}]} & = & v_{1}\\
y_{2,[0,\kappa_{2}-1]} & = & \varphi_{2,[0,\kappa_{2}-1]}(\zeta_{[-q_{1},-1]},x,v_{1,[0,r_{1}-\kappa_{1}-1]})\\
y_{2,[\kappa_{2}]} & = & v_{2}\\
 & \vdots\\
y_{s-1,[0,\kappa_{s-1}-1]} & = & \varphi_{s-1,[0,\kappa_{s-1}-1]}(\zeta_{[-q_{1},-1]},x,v_{1,[0,r_{1}-\kappa_{1}-1]},\ldots,v_{s-2,[0,r_{s-2}-\kappa_{s-2}-1]})\\
y_{s-1,[\kappa_{s-1}]} & = & v_{s-1}\\
y_{s,[0,\kappa_{s}-1]} & = & \varphi_{s,[0,\kappa_{s}-1]}(\zeta_{[-q_{1},-1]},x,v_{1,[0,r_{1}-\kappa_{1}-1]},\ldots,v_{s-1,[0,r_{s-1}-\kappa_{s-1}-1]})\\
y_{s,[\kappa_{s}]} & = & v_{s}\,,
\end{array}\label{eq:new_coordinates_2}
\end{equation}
where $r_{i}=(r_{i}^{1},\ldots,r_{i}^{m_{i}})$ denotes the components
of the multi-index $R$ corresponding to the components $y_{i}=(y_{i}^{1},\ldots,y_{i}^{m_{i}})$
of the flat output.

\end{thm}

\begin{proof}
In every step $i\geq1$ of the procedure, it is possible to forward-shift
the remaining components $\varphi_{rest_{i-1}}$ of the flat output
until every component depends explicitly on one of the remaining inputs
$u_{rest_{i-1}}$ ($\varphi_{rest_{0}}=\varphi$ and $u_{rest_{0}}=u$
for $i=1$). Otherwise, the property that all forward-shifts of a
flat output up to arbitrary order are functionally independent could
not hold. Since in every step we have $\mathrm{rank}(\partial_{u_{rest_{i-1}}}\varphi_{rest_{i-1},[K_{i}]})\geq1$,
at least one of the original inputs $u$ can be eliminated, and the
procedure terminates after at most $\dim(u)=m$ steps.

Now let us prove that $\kappa\leq R$. If the parameterization of
$x$ and $\zeta_{[-q_{1},-1]}$ by the flat output is substituted
into (\ref{eq:new _coordinates}) and $y_{[\kappa+\alpha]}$, $\alpha\geq0$
renamed according to $y_{[\kappa+\alpha]}=v_{[\alpha]}$, then the
equations
\begin{equation}
y_{i,[0,\kappa_{i}-1]}^{j_{i}}=\varphi_{i,[0,\kappa_{i}-1]}^{j_{i}}(\zeta_{[-q_{1},-1]},x,v_{1},v_{1,[1]},\ldots,v_{i-1},v_{i-1,[1]},\ldots)\,,\quad j_{i}=1,\ldots,m_{i}\,,\quad i=1,\ldots,s\label{eq:identity}
\end{equation}
must be satisfied identically. Since both $x$ and $\zeta_{[-q_{1},-1]}$
depend only on forward-shifts of the flat output up to the order $R-1$,
and the forward-shifts $y_{i,[0,\kappa_{i}-1]}^{j_{i}}$ of $y_{i}^{j_{i}}$
on the left-hand side of (\ref{eq:identity}) are not contained in
the quantities $v_{1},v_{1,[1]},\ldots,v_{i-1},v_{i-1,[1]},\ldots$
on the right-hand side, this can only hold in the case $\kappa\leq R$.
By the same argument, it is clear that (\ref{eq:new _coordinates})
is actually of the form (\ref{eq:new_coordinates_2}).

To prove $\#\kappa\geq n$, recall that there exist exactly $n$ independent
linear combinations of the differentials of a flat output and its
forward-shifts which are contained in $\mathrm{span}\{\mathrm{d}x\}$.
In the coordinates constructed during the above procedure, the forward-shifts
of the flat output up to the order $\kappa$ are given by the expressions
in (\ref{eq:new _coordinates}), and the higher forward-shifts are
forward-shifts of $v$. Thus, there can exist at most $\#\kappa$
independent linear combinations which are contained in $\mathrm{span}\{\mathrm{d}x\}$,
and hence $\#\kappa\geq n$. If the flat output (\ref{eq:flat_output_xu})
is independent of $\zeta_{[-q_{1}]},\dots,\zeta_{[-1]}$, then all
expressions in (\ref{eq:new _coordinates}) are independent of these
variables. Consequently, there exist exactly $\#\kappa$ independent
linear combinations of the differentials of the flat output and its
forward-shifts which are contained in $\mathrm{span}\{\mathrm{d}x\}$,
and hence $\#\kappa=n$.

To prove the relation with Theorem \ref{thm:Conditions_Input}, we
use again the representation (\ref{eq:new _coordinates}). In these
coordinates, it can be immediately observed that with $A=\kappa$
the differentials (\ref{eq:Condition_Input}) are linearly independent.
Moreover, it can also be seen that there exist exactly $\#\kappa$
independent linear combinations of the differentials of the flat output
and its forward-shifts which are contained in $\mathrm{span}\{\mathrm{d}\zeta_{[-q_{1}]},\ldots,\mathrm{d}\zeta_{[-1]},\mathrm{d}x\}$.
If there exists a multi-index $A$ such that the differentials (\ref{eq:Condition_Input})
are linearly independent, then there can exist at most $\#A$ independent
linear combinations of the differentials of the flat output and its
forward-shifts which are contained in $\mathrm{span}\{\mathrm{d}\zeta_{[-q_{1}]},\ldots,\mathrm{d}\zeta_{[-1]},\mathrm{d}x\}$.
Thus, the existence of such a multi-index with $\#A<\#\kappa$ would
be a contradiction.
\end{proof}
Since the condition of Theorem \ref{thm:Conditions_Input} is met,
the forward-shifts $y_{[\kappa]}$ of the flat output (\ref{eq:flat_output_xu})
can be introduced as a new input $v$ by a (dynamic) feedback according
to Theorem \ref{thm:exact_linearization}, where the controller state
$z$ corresponds to suitable forward-shifts of the flat output that
are contained in $y_{[0,\kappa-1]}$. However, the representation
(\ref{eq:new_coordinates_2}) of $y_{[0,\kappa-1]}$ offers a convenient
alternative. With the map $y_{[0,R]}=\phi(\zeta_{[-q_{1},-1]},x,v_{[0,R-\kappa]})$
defined by
\[
\begin{array}{ccl}
y_{1,[0,\kappa_{1}-1]} & = & \varphi_{1,[0,\kappa_{1}-1]}(\zeta_{[-q_{1},-1]},x)\\
y_{1,[\kappa_{1},r_{1}]} & = & v_{1,[0,r_{1}-\kappa_{1}]}\\
 & \vdots\\
y_{s,[0,\kappa_{s}-1]} & = & \varphi_{s,[0,\kappa_{s}-1]}(\zeta_{[-q_{1},-1]},x,v_{1,[0,r_{1}-\kappa_{1}-1]},\ldots,v_{s-1,[0,r_{s-1}-\kappa_{s-1}-1]})\\
y_{s,[\kappa_{s},r_{s}]} & = & v_{s,[0,r_{s}-\kappa_{s}]}\,,
\end{array}
\]
we can formulate the following corollary.
\begin{cor}
\label{cor:quasistatic_linearization}A flat system (\ref{eq:sys})
with a flat output of the form (\ref{eq:flat_output_xu}) can be exactly
linearized with respect to this flat output by a feedback of the form
\begin{equation}
u=F_{u}\circ\phi(\zeta_{[-q_{1},-1]},x,v_{[0,R-\kappa]})\,,\label{eq:alternative_linearizing_feedback}
\end{equation}
such that the input-output behaviour of the closed-loop system is
given by $y_{[\kappa]}=v$. If the flat output (\ref{eq:flat_output_xu})
is independent of the variables $\zeta_{[-q_{1}]},\dots,\zeta_{[-1]}$,
then the feedback has the form
\begin{equation}
u=F_{u}\circ\phi(x,v_{[0,R-\kappa]})\,,\label{eq:quasistatic_feedback}
\end{equation}
and $\#\kappa=n$.
\end{cor}

In contrast to a dynamic feedback (\ref{eq:Linearizing_Feedback})
with the controller state $z$, the feedback (\ref{eq:alternative_linearizing_feedback})
depends only on the state $x$ and past values $\zeta_{[-q_{1}]},\dots,\zeta_{[-1]}$
of the system trajectory. Since the values of $\zeta_{[-q_{1}]},\dots,\zeta_{[-1]}$
are available anyway from past measurements and/or past control inputs
(depending on the choice of $\zeta$, cf. (\ref{eq:sys_extension})),
the implementation of a feedback (\ref{eq:alternative_linearizing_feedback})
is straightforward. However, even though there is no dedicated controller
dynamics as in (\ref{eq:Linearizing_Feedback}), the required past
values $\zeta_{[-q_{1}]},\dots,\zeta_{[-1]}$ have to be stored. Thus,
the feedback (\ref{eq:alternative_linearizing_feedback}) can be considered
either as a special case of a dynamic feedback or a generalization
of the class of discrete-time quasi-static feedbacks (as they are
defined in \cite{ArandaKotta:2001}) to backward-shifts $\dots,\zeta_{[-2]},\zeta_{[-1]}$
of the system variables. A feedback of the special form (\ref{eq:quasistatic_feedback})
has been used in \cite{DiwoldKolarSchoberl:2022} for the exact linearization
of the discrete-time model of a gantry crane.
\begin{rem}
If the flat output (\ref{eq:flat_output_xu}) is
of the form $y=\varphi(x)$ -- i.e., independent of the input $u$
as well as past values $\zeta_{[-q_{1}]},\dots,\zeta_{[-1]}$ of the
system trajectory -- then the proposed procedure for the construction
of a minimal multi-index $\kappa$ is from a technical point of view
similar to the inversion algorithm stated e.g. in \cite{Kotta:1995},
see also \cite{KottaNijmeijer:1991} or \cite{Kotta:1990}. In the
inversion algorithm, which deals with the forward right-invertibility
of a system (\ref{eq:sys}) with a (not necessarily flat) output $y=\varphi(x)$
and possibly also $\dim(y)\neq\dim(u)$, the components of the output
are also shifted until the Jacobian matrices with respect to the input
variables meet certain rank conditions. However, it should be noted
that in every step of the inversion algorithm only one-fold shifts
of the components of the output are performed. Furthermore, it should
also be noted that the so-called invertibility indices computed by
the inversion algorithm are related to but not the same as the components
of the constructed minimal multi-index $\kappa$.
\end{rem}

\section{\label{sec:tracking_control}Tracking control design}

Like in the continuous-time case, the exact linearization can be used
as a first step in the design of a flatness-based tracking control.
For an exact linearization according to Theorem \ref{thm:exact_linearization}
with the choice $A=R$, which is always possible, the closed-loop
system has the form of a classical state representation. Thus, the
design of a tracking control is straightforward. For $A\leq R$, however,
the closed-loop system depends also on forward-shifts $v_{[0,R-A]}$
of the new input $v$. Thus, when designing a control law for $v$,
also the corresponding expressions for these forward-shifts have to
be derived. For a discussion of this problem in the continuous-time
case see e.g. \cite{DelaleauRudolph:1998}, \cite{Rudolph:2021},
or \cite{GstoettnerKolarSchoberl:2021}.

In the following, we demonstrate the design of a tracking control
for flat outputs of the form (\ref{eq:flat_output_xu}) and an exact
linearization by a feedback (\ref{eq:alternative_linearizing_feedback})
according to Corollary \ref{cor:quasistatic_linearization}. We assume
that the multi-index $\kappa$ which determines the new input $v$
has been constructed in accordance with the procedure of Section \ref{subsec:quasistatic_linearization_xu_flat},
and make use of the corresponding notation. With the control law
\begin{equation}
v_{i}^{j_{i}}=y_{i,[\kappa_{i}^{j_{i}}]}^{j_{i},d}-\sum_{\beta=0}^{\kappa_{i}^{j_{i}}-1}a_{i}^{j_{i},\beta}(y_{i,[\beta]}^{j_{i}}-y_{i,[\beta]}^{j_{i},d})\,,\quad j_{i}=1,\ldots,m_{i}\,,\quad i=1,\ldots,s\label{eq:stabilizing_feedback}
\end{equation}
for the exactly linearized system $y_{[\kappa]}=v$, the tracking
error $e_{i}^{j_{i}}=y_{i}^{j_{i}}-y_{i}^{j_{i},d}$ with respect
to an arbitrary reference trajectory $y_{i}^{j_{i},d}(k)$ is subject
to the tracking error dynamics
\begin{equation}
e_{i,[\kappa_{i}^{j_{i}}]}^{j_{i}}+\sum_{\beta=0}^{\kappa_{i}^{j_{i}}-1}a_{i}^{j_{i},\beta}e_{i,[\beta]}^{j_{i}}=0\,,\quad j_{i}=1,\ldots,m_{i}\,,\quad i=1,\ldots,s\,.\label{eq:tracking_error_dynamics}
\end{equation}
The eigenvalues of the $\sum_{i=1}^{s}m_{i}=m$ decoupled tracking
error systems (\ref{eq:tracking_error_dynamics}) can be placed arbitrarily
by a suitable choice of the coefficients $a_{i}^{j_{i},\beta}\in\mathbb{R}$.
The forward-shifts $v_{[0,R-\kappa]}$ that are needed in the linearizing
feedback (\ref{eq:alternative_linearizing_feedback}) can be determined
by shifting (\ref{eq:stabilizing_feedback}) and using $y_{i,[\kappa_{i}+\gamma]}^{j_{i}}=v_{i,[\gamma]}^{j_{i}}$,
$\gamma\geq0$, which leads to equations of the form
\begin{equation}
v_{i,[\gamma]}^{j_{i}}=y_{i,[\kappa_{i}^{j_{i}}+\gamma]}^{j_{i},d}-\sum_{\alpha=\kappa_{i}^{j_{i}}}^{\kappa_{i}^{j_{i}}-1+\gamma}a_{i}^{j_{i},\alpha-\gamma}(v_{i,[\alpha-\kappa_{i}^{j_{i}}]}^{j_{i}}-y_{i,[\alpha]}^{j_{i},d})-\sum_{\beta=\gamma}^{\kappa_{i}^{j_{i}}-1}a_{i}^{j_{i},\beta-\gamma}(y_{i,[\beta]}^{j_{i}}-y_{i,[\beta]}^{j_{i},d})\,.\label{eq:stabilizing_feedback_shifted}
\end{equation}
Since the future values $y_{[0,\kappa-1]}$ of the flat output which
appear in (\ref{eq:stabilizing_feedback}) and (\ref{eq:stabilizing_feedback_shifted})
are in general not available as measurements, we use again the expressions
(\ref{eq:new_coordinates_2}) and finally obtain the system of equations
\begin{align}
v_{1}^{j_{1}} & =y_{1,[\kappa_{1}^{j_{1}}]}^{j_{1},d}-\sum_{\beta=0}^{\kappa_{1}^{j_{1}}-1}a_{1}^{j_{1},\beta}(\varphi_{1,[\beta]}^{j_{1}}-y_{1,[\beta]}^{j_{1},d})\nonumber \\
v_{1,[1]}^{j_{1}} & =y_{1,[\kappa_{1}^{j_{1}}+1]}^{j_{1},d}-a_{1}^{j_{1},\kappa_{1}^{j_{1}}-1}(v_{1}^{j_{1}}-y_{1,[\kappa_{1}^{j_{1}}]}^{j_{1},d})-\sum_{\beta=1}^{\kappa_{1}^{j_{1}}-1}a_{1}^{j_{1},\beta-1}(\varphi_{1,[\beta]}^{j_{1}}-y_{1,[\beta]}^{j_{1},d})\nonumber \\
 & \:\:\vdots\nonumber \\
v_{1,[r_{1}^{j_{1}}-\kappa_{1}^{j_{1}}]}^{j_{1}} & =y_{1,[r_{1}^{j_{1}}]}^{j_{1},d}-\sum_{\alpha=\kappa_{1}^{j_{1}}}^{r_{1}^{j_{1}}-1}a_{1}^{j_{1},\alpha-r_{1}^{j_{1}}+\kappa_{1}^{j_{1}}}(v_{1,[\alpha-\kappa_{1}^{j_{1}}]}^{j_{1}}-y_{1,[\alpha]}^{j_{1},d})-\hspace{-8pt}\sum_{\beta=r_{1}^{j_{1}}-\kappa_{1}^{j_{1}}}^{\kappa_{1}^{j_{1}}-1}a_{1}^{j_{1},\beta-r_{1}^{j_{1}}+\kappa_{1}^{j_{1}}}(\varphi_{1,[\beta]}^{j_{1}}-y_{1,[\beta]}^{j_{1},d})\nonumber \\
 & \:\:\vdots\label{eq:final_equ_sys}\\
v_{s}^{j_{s}} & =y_{s,[\kappa_{s}^{j_{s}}]}^{j_{s},d}-\sum_{\beta=0}^{\kappa_{s}^{j_{s}}-1}a_{s}^{j_{s},\beta}(\varphi_{s,[\beta]}^{j_{s}}-y_{s,[\beta]}^{j_{s},d})\nonumber \\
v_{s,[1]}^{j_{s}} & =y_{s,[\kappa_{s}^{j_{s}}+1]}^{j_{s},d}-a_{s}^{j_{s},\kappa_{s}^{j_{s}}-1}(v_{s}^{j_{s}}-y_{s,[\kappa_{s}^{j_{s}}]}^{j_{s},d})-\sum_{\beta=1}^{\kappa_{s}^{j_{s}}-1}a_{s}^{j_{s},\beta-1}(\varphi_{s,[\beta]}^{j_{s}}-y_{s,[\beta]}^{j_{s},d})\nonumber \\
 & \:\:\vdots\nonumber \\
v_{s,[r_{s}^{j_{s}}-\kappa_{s}^{j_{s}}]}^{j_{s}} & =y_{s,[r_{s}^{j_{s}}]}^{j_{s},d}-\sum_{\alpha=\kappa_{s}^{j_{s}}}^{r_{s}^{j_{s}}-1}a_{s}^{j_{s},\alpha-r_{s}^{j_{s}}+\kappa_{s}^{j_{s}}}(v_{s,[\alpha-\kappa_{s}^{j_{s}}]}^{j_{s}}-y_{s,[\alpha]}^{j_{s},d})-\hspace{-8pt}\sum_{\beta=r_{s}^{j_{s}}-\kappa_{s}^{j_{s}}}^{\kappa_{s}^{j_{s}}-1}a_{s}^{j_{s},\beta-r_{s}^{j_{s}}+\kappa_{s}^{j_{s}}}(\varphi_{s,[\beta]}^{j_{s}}-y_{s,[\beta]}^{j_{s},d})\,.\nonumber 
\end{align}
Because of the triangular dependence of the functions $\varphi_{1,[0,\kappa_{1}-1]},\varphi_{2,[0,\kappa_{2}-1]},\ldots$
of (\ref{eq:new_coordinates_2}) on the variables $v_{1,[0,r_{1}-\kappa_{1}-1]}$,
$v_{2,[0,r_{2}-\kappa_{2}-1]}$, $\ldots$, the equations (\ref{eq:final_equ_sys})
have a triangular structure and can be solved systematically from
top to bottom for all elements of $v_{[0,R-\kappa]}$ as a function
of $\zeta_{[-q_{1},-1]}$, $x$, and the reference trajectory $y_{[0,R]}^{d}$,
i.e.,
\begin{equation}
v_{[0,R-\kappa]}=\rho(\zeta_{[-q_{1},-1]},x,y_{[0,R]}^{d})\,.\label{eq:final_equ_sys_sol}
\end{equation}
Substituting (\ref{eq:final_equ_sys_sol}) into the linearizing feedback
(\ref{eq:alternative_linearizing_feedback}) yields a tracking control
law of the form 
\begin{align*}
u & =\eta(\zeta_{[-q_{1},-1]},x,y_{[0,R]}^{d})\,.
\end{align*}
Besides the known reference trajectory $y_{[0,R]}^{d}$, this tracking
control law depends like the linearizing feedback (\ref{eq:alternative_linearizing_feedback})
only on the state $x$ and past values $\zeta_{[-q_{1}]},\dots,\zeta_{[-1]}$
of the system trajectory. Thus, if the state $x$ of the system (\ref{eq:sys})
can be measured, an implementation is again straightforward.
\begin{rem}
If all coefficients $a_{i}^{j_{i},\beta}$ in the
control law (\ref{eq:stabilizing_feedback}) are set to zero, all
eigenvalues of the tracking error dynamics (\ref{eq:tracking_error_dynamics})
are located at the origin of the complex plane and a dead-beat control
is obtained. In this case, the system of equations (\ref{eq:final_equ_sys})
drastically simplifies. Because of $v_{i}^{j_{i}}=y_{i,[\kappa_{i}^{j_{i}}]}^{j_{i},d}$,
the forward-shifts of $v$ required in the linearizing feedback (\ref{eq:alternative_linearizing_feedback})
are simply higher-order forward-shifts of the reference trajectory
$y^{d}$.
\end{rem}

\section{\label{sec:example}Examples}

As already mentioned in the introduction, an important application
for the concept of discrete-time flatness are discretized continuous-time
systems. In the following, we illustrate our results by the discretized
models of a wheeled mobile robot and a 3DOF helicopter.

\subsection{\label{subsec:Wheeled_mobile_robot}Wheeled mobile robot}

As first example let us consider a wheeled mobile robot, which has
already been studied in the context of discrete-time dynamic feedback
linearization in \cite{OroscoVelascoAranda:2004} and \cite{Aranda-BricaireMoog:2008}.
The continuous-time system is given by
\begin{equation}
\begin{aligned}\dot{x}^{1} & =u^{1}\cos(x^{3})\\
\dot{x}^{2} & =u^{1}\sin(x^{3})\\
\dot{x}^{3} & =u^{2}\,,
\end{aligned}
\label{eq:mobile_robot_cont}
\end{equation}
and is also known as kinematic car model, see e.g. \cite{NijmeijervanderSchaft:1990}.
The state variables $x^{1}$ and $x^{2}$ describe the position and
$x^{3}$ the angle of the mobile robot. The control inputs $u^{1}$
and $u^{2}$ represent the translatory and the angular velocity. An
exact discretization of the system (\ref{eq:mobile_robot_cont}) with
the sampling time $T>0$ yields the discrete-time system
\begin{equation}
\begin{aligned}x^{1,+} & =x^{1}+u^{1}T\cos\left(x^{3}+u^{2}\tfrac{T}{2}\right)\tfrac{\sin\left(u^{2}\tfrac{T}{2}\right)}{u^{2}\tfrac{T}{2}}\\
x^{2,+} & =x^{2}+u^{1}T\sin\left(x^{3}+u^{2}\tfrac{T}{2}\right)\tfrac{\sin\left(u^{2}\tfrac{T}{2}\right)}{u^{2}\tfrac{T}{2}}\\
x^{3,+} & =x^{3}+u^{2}T\,,
\end{aligned}
\label{eq:mobile_robot_discrete_1}
\end{equation}
cf. \cite{OroscoVelascoAranda:2004} or \cite{Sira-RamirezRouchon:2003}.
With the input transformation
\begin{equation}
\begin{array}{ccl}
\bar{u}^{1} & = & 2u^{1}\tfrac{\sin\left(u^{2}\tfrac{T}{2}\right)}{u^{2}}\\
\bar{u}^{2} & = & x^{3}+u^{2}\tfrac{T}{2}\,,
\end{array}\label{eq:mobile_robot_input_transformation}
\end{equation}
this system can be simplified to
\begin{equation}
\begin{aligned}x^{1,+} & =x^{1}+\bar{u}^{1}\cos(\bar{u}^{2})\\
x^{2,+} & =x^{2}+\bar{u}^{1}\sin(\bar{u}^{2})\\
x^{3,+} & =2\bar{u}^{2}-x^{3}\,.
\end{aligned}
\label{eq:mobile_robot_discrete_2}
\end{equation}
As shown in \cite{KolarSchoberlDiwold:2019}, the system (\ref{eq:mobile_robot_discrete_2})
is not forward-flat. However, it is flat in the more general sense
of Definition \ref{def:Flatness}. With the choice
\[
\begin{aligned}\zeta^{1} & =x^{3}\\
\zeta^{2} & =x^{1}
\end{aligned}
\]
for the variables $\zeta$ according to (\ref{eq:sys_extension}),
a flat output is given by
\begin{equation}
y=(\zeta_{[-1]}^{1},x^{1}\sin\left(\tfrac{\zeta_{[-1]}^{1}+x^{3}}{2}\right)-x^{2}\cos\left(\tfrac{\zeta_{[-1]}^{1}+x^{3}}{2}\right))\,.\label{eq:mobile_robot_discrete_flat_output}
\end{equation}
The corresponding parameterization (\ref{eq:flat_parameterization})
of the system variables has the form
\begin{equation}
\begin{aligned}x & =F_{x}(y^{1},y^{2},\dots,y_{[2]}^{1},y_{[1]}^{2})\\
\bar{u} & =F_{\bar{u}}(y^{1},y^{2},\dots,y_{[3]}^{1},y_{[2]}^{2})\,.
\end{aligned}
\label{eq:mobile_robot_parameterization}
\end{equation}
Thus, the orders of the highest forward-shifts of the flat output
that appear in (\ref{eq:flat_parameterization}) are given by $R=(3,2)$,
and an exact linearization by a dynamic feedback which leads to an
input-output behaviour
\begin{equation}
\begin{array}{ccl}
y_{[3]}^{1} & = & v^{1}\\
y_{[2]}^{2} & = & v^{2}
\end{array}\label{eq:mobile_robot_dynamic_linearization}
\end{equation}
is possible with the standard approach discussed in \cite{DiwoldKolarSchoberl:2020}. 

In the following, we investigate whether also lower-order forward-shifts
of the flat output (\ref{eq:mobile_robot_discrete_flat_output}) can
be chosen as a new input. Since the flat output is of the form (\ref{eq:flat_output_xu}),
we can apply the procedure of Section \ref{subsec:quasistatic_linearization_xu_flat}
and use the corresponding notation. In the first step, both components
of the flat output have to be shifted until they depend explicitly
on the input $\bar{u}$. Because of
\begin{align*}
\varphi^{1} & =\zeta_{[-1]}^{1}\\
\delta(\varphi^{1}) & =x^{3}\\
\delta^{2}(\varphi^{1}) & =2\bar{u}^{2}-x^{3}
\end{align*}
and
\begin{align*}
\varphi^{2} & =x^{1}\sin\left(\tfrac{\zeta_{[-1]}^{1}+x^{3}}{2}\right)-x^{2}\cos\left(\tfrac{\zeta_{[-1]}^{1}+x^{3}}{2}\right)\\
\delta(\varphi^{2}) & =x^{1}\sin(\bar{u}^{2})-x^{2}\cos(\bar{u}^{2})\,,
\end{align*}
this is the case for the second and the first forward-shift, respectively.
Hence, we obtain $K_{1}=(2,1)$, and since $\varphi_{[K_{1}]}$ is
independent of $\bar{u}^{1}$ we clearly have $m_{1}=\mathrm{rank}(\partial_{\bar{u}}\varphi_{[K_{1}]})=1$.\footnote{In contrast to Section \ref{subsec:quasistatic_linearization_xu_flat},
for the sake of simplicity we do not renumber the components $\bar{u}^{1}$
and $\bar{u}^{2}$ of the input.} At this point, we can choose whether we introduce $\delta^{2}(\varphi^{1})$
or $\delta(\varphi^{2})$ as new input $v_{1}=v_{1}^{1}$.\footnote{To emphasize that $v_{1}$ could in general consist of more than one
component we write $v_{1}^{1}$.} In the following, we proceed with $\varphi_{1}=\varphi^{1}$ and
$\varphi_{rest_{1}}=\varphi^{2}$. Consequently, we get $\kappa_{1}=k_{1}^{1}=2$
and $K_{rest_{1}}=k_{1}^{2}=1$. After the coordinate transformation
\begin{align*}
v_{1}^{1} & =\varphi_{1,[2]}^{1}=2\bar{u}^{2}-x^{3}\\
v_{1,[1]}^{1} & =\varphi_{1,[3]}^{1}=2\bar{u}_{[1]}^{2}-2\bar{u}^{2}+x^{3}\,,\\
 & \:\:\vdots
\end{align*}
which replaces $\bar{u}^{2}$ and its forward-shifts by $v_{1}^{1}$
and its forward-shifts, we have
\begin{align*}
y_{1,[0,\kappa_{1}-1]} & =\begin{bmatrix}\varphi_{1}^{1}\\
\varphi_{1,[1]}^{1}
\end{bmatrix}=\begin{bmatrix}\zeta_{[-1]}^{1}\\
x^{3}
\end{bmatrix}\\
y_{1,[\kappa_{1}]} & =\varphi_{1,[2]}^{1}=v_{1}^{1}\\
y_{rest_{1},[0,K_{rest_{1}}-1]} & =\varphi_{rest_{1}}^{1}=x^{1}\sin\left(\tfrac{\zeta_{[-1]}^{1}+x^{3}}{2}\right)-x^{2}\cos\left(\tfrac{\zeta_{[-1]}^{1}+x^{3}}{2}\right)\\
y_{rest_{1},[K_{rest_{1}}]} & =\varphi_{rest_{1},[1]}^{1}=x^{1}\sin\left(\tfrac{x^{3}+v_{1}^{1}}{2}\right)-x^{2}\cos\left(\tfrac{x^{3}+v_{1}^{1}}{2}\right)\,.
\end{align*}
In the second step, we have to shift the remaining component $\varphi_{rest_{1}}=\varphi^{2}$
of the flat output until it depends explicitly on the remaining input
$\bar{u}^{1}$. This is the case for its second forward-shift
\begin{equation}
\delta^{2}(\varphi_{rest_{1}}^{1})=\left(x^{1}+\bar{u}^{1}\cos\left(\tfrac{x^{3}+v_{1}^{1}}{2}\right)\right)\sin\left(\tfrac{v_{1}^{1}+v_{1,[1]}^{1}}{2}\right)-\left(x^{2}+\bar{u}^{1}\sin\left(\tfrac{x^{3}+v_{1}^{1}}{2}\right)\right)\cos\left(\tfrac{v_{1}^{1}+v_{1,[1]}^{1}}{2}\right)\,.\label{eq:mobile_robot_varphi_K2}
\end{equation}
Because of $\mathrm{rank}(\partial_{\bar{u}^{1}}\varphi_{rest_{1},[2]}^{1})=\dim(\bar{u}^{1})=1$,
the procedure terminates with $\varphi_{2}=\varphi_{rest_{1}}$ and
$\kappa_{2}=K_{2}=2$. After introducing (\ref{eq:mobile_robot_varphi_K2})
as new input $v_{2}=v_{2}^{1}$, the forward-shifts of the flat output
(\ref{eq:mobile_robot_discrete_flat_output}) up to the orders $\kappa=(\kappa_{1},\kappa_{2})=(2,2)$
are given by
\begin{align}
y_{1,[0,\kappa_{1}-1]} & =\begin{bmatrix}\varphi_{1}^{1}\\
\varphi_{1,[1]}^{1}
\end{bmatrix}=\begin{bmatrix}\zeta_{[-1]}^{1}\\
x^{3}
\end{bmatrix}\nonumber \\
y_{1,[\kappa_{1}]} & =\varphi_{1,[2]}^{1}=v_{1}^{1}\nonumber \\
y_{2,[0,\kappa_{2}-1]} & =\begin{bmatrix}\varphi_{2}^{1}\\
\varphi_{2,[1]}^{1}
\end{bmatrix}=\begin{bmatrix}x^{1}\sin\left(\tfrac{\zeta_{[-1]}^{1}+x^{3}}{2}\right)-x^{2}\cos\left(\tfrac{\zeta_{[-1]}^{1}+x^{3}}{2}\right)\\
x^{1}\sin\left(\frac{x^{3}+v_{1}^{1}}{2}\right)-x^{2}\cos\left(\frac{x^{3}+v_{1}^{1}}{2}\right)
\end{bmatrix}\label{eq:mobile_robot_new coordinates}\\
y_{2,[\kappa_{2}]} & =\varphi_{2,[2]}^{1}=v_{2}^{1}\,.\nonumber 
\end{align}
Substituting (\ref{eq:mobile_robot_new coordinates}) as well as $y_{1,[3]}^{1}=v_{1,[1]}^{1}$
into the parameterization (\ref{eq:mobile_robot_parameterization})
of the control input $\bar{u}$ yields a feedback of the form
\begin{equation}
\bar{u}=F_{\bar{u}}\circ\phi(\zeta_{[-1]}^{1},x^{1},x^{2},x^{3},v_{1}^{1},v_{1,[1]}^{1},v_{2}^{1})\label{eq:mobile_robot_linearizing_feedback}
\end{equation}
such that the input-output behaviour of the closed-loop system is
given by
\begin{equation}
\begin{array}{ccl}
y_{1,[2]}^{1} & = & v_{1}^{1}\\
y_{2,[2]}^{1} & = & v_{2}^{1}\,,
\end{array}\label{eq:mobile_robot_exactly_linearized}
\end{equation}
cf. Corollary \ref{cor:quasistatic_linearization}. Thus, in contrast
to (\ref{eq:mobile_robot_dynamic_linearization}), we can actually
use the second instead of the third forward-shift of the first component
of the flat output as a new input.

For the exactly linearized system (\ref{eq:mobile_robot_exactly_linearized}),
the design of a tracking control is now straightforward. The control
law
\begin{align*}
v_{1}^{1} & =y_{1,[2]}^{1,d}-a_{1}^{1,1}(y_{1,[1]}^{1}-y_{1,[1]}^{1,d})-a_{1}^{1,0}(y_{1}^{1}-y_{1}^{1,d})\\
v_{2}^{1} & =y_{2,[2]}^{1,d}-a_{2}^{1,1}(y_{2,[1]}^{1}-y_{2,[1]}^{1,d})-a_{2}^{1,0}(y_{2}^{1}-y_{2}^{1,d})
\end{align*}
results in the linear tracking error dynamics
\begin{align*}
e_{1,[2]}^{1}+a_{1}^{1,1}e_{1,[1]}^{1}+a_{1}^{1,0}e_{1}^{1} & =0\\
e_{2,[2]}^{1}+a_{2}^{1,1}e_{2,[1]}^{1}+a_{2}^{1,0}e_{2}^{1} & =0
\end{align*}
with a total order of $\#\kappa=2+2=4$ instead of $\#R=3+2=5$ with
the standard approach. Since the linearizing feedback (\ref{eq:mobile_robot_linearizing_feedback})
depends also on $v_{1,[1]}^{1}$, the system of equations (\ref{eq:final_equ_sys})
is given by
\begin{align*}
v_{1}^{1} & =y_{1,[2]}^{1,d}-a_{1}^{1,1}(\varphi_{1,[1]}^{1}-y_{1,[1]}^{1,d})-a_{1}^{1,0}(\varphi_{1}^{1}-y_{1}^{1,d})\\
v_{1,[1]}^{1} & =y_{1,[3]}^{1,d}-a_{1}^{1,1}(v_{1}^{1}-y_{1,[2]}^{1,d})-a_{1}^{1,0}(\varphi_{1,[1]}^{1}-y_{1,[1]}^{1,d})\\
v_{2}^{1} & =y_{2,[2]}^{1,d}-a_{2}^{1,1}(\varphi_{2,[1]}^{1}-y_{2,[1]}^{1,d})-a_{2}^{1,0}(\varphi_{2}^{1}-y_{2}^{1,d})
\end{align*}
with the functions $\varphi_{1}^{1},\varphi_{1,[1]}^{1},\varphi_{2}^{1},\varphi_{2,[1]}^{1}$
according to (\ref{eq:mobile_robot_new coordinates}). This system
of equations can be solved from top to bottom for $v_{1}^{1},v_{1,[1]}^{1}$,
and $v_{2}^{1}$ as a function of $\zeta_{[-1]}^{1}$, $x$, and the
reference trajectory $y_{[0,R]}^{d}$. Substituting the solution into
the linearizing feedback (\ref{eq:mobile_robot_linearizing_feedback})
yields a control law of the form
\begin{align*}
\bar{u}^{1} & =\eta^{1}(\zeta_{[-1]}^{1},x^{1},x^{2},x^{3},y_{1,[0,3]}^{1,d},y_{2,[0,2]}^{1,d})\\
\bar{u}^{2} & =\eta^{2}(\zeta_{[-1]}^{1},x^{1},x^{2},x^{3},y_{1,[0,3]}^{1,d},y_{2,[0,2]}^{1,d})\,.
\end{align*}
With the inverse of the input transformation (\ref{eq:mobile_robot_input_transformation}),
the corresponding control law for the system (\ref{eq:mobile_robot_discrete_1})
or (\ref{eq:mobile_robot_cont}) with the original inputs $u^{1}$
and $u^{2}$ follows. The presence of the variable $\zeta_{[-1]}^{1}$
is no obstacle for a practical implementation, since it simply represents
a past value of $x^{3}$.

\subsection{3DOF helicopter}

As a second example, let us consider the three-degrees-of-freedom
helicopter laboratory experiment of \cite{KieferKugiKemmetmueller:2004}.
The continuous-time system is given by
\begin{equation}
\begin{array}{ccl}
\dot{q}^{1} & = & \omega^{1}\\
\dot{q}^{2} & = & \omega^{2}\\
\dot{q}^{3} & = & \omega^{3}\\
\dot{\omega}^{1} & = & b_{1}\cos(q^{2})\sin(q^{3})u^{1}\\
\dot{\omega}^{2} & = & a_{1}\sin(q^{2})+a_{2}\cos(q^{2})+b_{2}\cos(q^{3})u^{1}\\
\dot{\omega}^{3} & = & a_{3}\cos(q^{2})\sin(q^{3})+b_{3}u^{2}
\end{array}\label{eq:helicopter_cont}
\end{equation}
with the travel angle $q^{1}$, the elevation angle $q^{2}$, and
the pitch angle $q^{3}$ as well as the corresponding angular velocities
$\omega^{1}$, $\omega^{2}$, and $\omega^{3}$. The control inputs
$u^{1}$ and $u^{2}$ are the sum and the difference of the thrusts
of the two propellers. The constant coefficients $a_{1}$, $a_{2}$,
$a_{3}$ and $b_{1},$ $b_{2}$ depend on the masses and the geometric
parameters. As shown in \cite{KieferKugiKemmetmueller:2004}, the
system (\ref{eq:helicopter_cont}) is flat and a flat output is given
by\footnote{The components of the flat output are already sorted in such a way
that we do not need a renumbering during our calculations.}
\begin{equation}
y=(q^{2},q^{1})\,.\label{eq:helicopter_flat_output}
\end{equation}
Since the system equations are significantly more complex than those
of the wheeled mobile robot of Section \ref{subsec:Wheeled_mobile_robot},
instead of an exact discretization we perform an approximate discretization
with the Euler-method. The resulting system
\begin{equation}
\begin{array}{ccl}
q^{1,+} & = & q^{1}+T\omega^{1}\\
q^{2,+} & = & q^{2}+T\omega^{2}\\
q^{3,+} & = & q^{3}+T\omega^{3}\\
\omega^{1,+} & = & \omega^{1}+Tb_{1}\cos(q^{2})\sin(q^{3})u^{1}\\
\omega^{2,+} & = & \omega^{2}+T\left(a_{1}\sin(q^{2})+a_{2}\cos(q^{2})+b_{2}\cos(q^{3})u^{1}\right)\\
\omega^{3,+} & = & \omega^{3}+T\left(a_{3}\cos(q^{2})\sin(q^{3})+b_{3}u^{2}\right)
\end{array}\label{eq:helicopter_disc}
\end{equation}
is forward-flat, and the flat output (\ref{eq:helicopter_flat_output})
is preserved.\footnote{Neither an exact nor an approximate discretization necessarily preserve
the flatness or the static feedback linearizability of a continuous-time
system, see e.g. \cite{DiwoldKolarSchoberl:2020}, \cite{DiwoldKolarSchoberl:2022},
or \cite{Grizzle:1986} for a further discussion.} This can be checked either with the systematic tests proposed in
\cite{KolarSchoberlDiwold:2019} and \cite{KolarDiwoldSchoberl:2019},
or by simply verifying that all state- and input variables of (\ref{eq:helicopter_disc})
can indeed be expressed by (\ref{eq:helicopter_flat_output}) and
its forward-shifts. The corresponding parameterization (\ref{eq:flat_parameterization})
is of the form
\begin{equation}
\begin{aligned}x & =F_{x}(y^{1},y^{2},\dots,y_{[3]}^{1},y_{[3]}^{2})\\
u & =F_{u}(y^{1},y^{2},\dots,y_{[4]}^{1},y_{[4]}^{2})
\end{aligned}
\label{eq:helicopter_parameterization}
\end{equation}
with the orders of the highest required forward-shifts given by $R=(4,4)$.
Thus, an exact linearization with a dynamic feedback according to
the standard approach discussed in \cite{DiwoldKolarSchoberl:2020}
would lead to an input-output behaviour
\begin{equation}
\begin{array}{ccl}
y_{[4]}^{1} & = & v^{1}\\
y_{[4]}^{2} & = & v^{2}\,.
\end{array}\label{eq:helicopter_dynamic_linearization}
\end{equation}

Since the flat output (\ref{eq:helicopter_flat_output}) is of the
form (\ref{eq:flat_output_xu}), we can again apply the procedure
of Section \ref{subsec:quasistatic_linearization_xu_flat} in order
to determine whether also lower-order forward-shifts can be used as
new input. In the first step, both components of the flat output have
to be shifted until they depend explicitly on the input $u$. Because
of
\begin{align*}
\varphi^{1} & =q^{2}\\
\delta(\varphi^{1}) & =q^{2}+T\omega^{2}\\
\delta^{2}(\varphi^{1}) & =q^{2}+2T\omega^{2}+T^{2}\left(a_{1}\sin(q^{2})+a_{2}\cos(q^{2})+b_{2}\cos(q^{3})u^{1}\right)
\end{align*}
and
\begin{align*}
\varphi^{2} & =q^{1}\\
\delta(\varphi^{2}) & =q^{1}+T\omega^{1}\\
\delta^{2}(\varphi^{2}) & =q^{1}+2T\omega^{1}+T^{2}b_{1}\cos(q^{2})\sin(q^{3})u^{1}\,,
\end{align*}
this is the case for the second forward-shifts. Hence, we obtain $K_{1}=(2,2)$,
and since $\varphi_{[K_{1}]}$ is independent of $u^{2}$ we obviously
have $m_{1}=\mathrm{rank}(\partial_{u}\varphi_{[K_{1}]})=1$. Again,
we can choose whether we introduce $\delta^{2}(\varphi^{1})$ or $\delta^{2}(\varphi^{2})$
as new input $v_{1}=v_{1}^{1}$. In the following we proceed with
$\varphi_{1}=\varphi^{1}$ and $\varphi_{rest_{1}}=\varphi^{2}$,
and get $\kappa_{1}=k_{1}^{1}=2$ as well as $K_{rest_{1}}=k_{1}^{2}=2$.
The other choice would lead to a feedback with a singularity for $q^{3}=0$,
which is not suitable for a practical application since the pitch
angle $q^{3}$ is zero in an equilibrium position. After the coordinate
transformation
\begin{align*}
v_{1}^{1} & =\varphi_{1,[2]}^{1}=q^{2}+2T\omega^{2}+T^{2}\left(a_{1}\sin(q^{2})+a_{2}\cos(q^{2})+b_{2}\cos(q^{3})u^{1}\right)\\
v_{1,[1]}^{1} & =\varphi_{1,[3]}^{1}(q^{2},q^{3},\omega^{2},\omega^{3},u^{1},u_{[1]}^{1})\,,\\
 & \:\:\vdots
\end{align*}
which replaces $u^{1}$ and its forward-shifts by $v_{1}^{1}$ and
its forward-shifts, we have
\begin{align*}
y_{1,[0,\kappa_{1}-1]} & =\begin{bmatrix}\varphi_{1}^{1}\\
\varphi_{1,[1]}^{1}
\end{bmatrix}=\begin{bmatrix}q^{2}\\
q^{2}+T\omega^{2}
\end{bmatrix}\\
y_{1,[\kappa_{1}]} & =\varphi_{1,[2]}^{1}=v_{1}^{1}\\
y_{rest_{1},[0,K_{rest_{1}}-1]} & =\begin{bmatrix}\varphi_{rest_{1}}^{1}\\
\varphi_{rest_{1},[1]}^{1}
\end{bmatrix}=\begin{bmatrix}q^{1}\\
q^{1}+T\omega^{1}
\end{bmatrix}\\
y_{rest_{1},[K_{rest_{1}}]} & =\varphi_{rest_{1},[2]}^{1}=q^{1}+2T\omega^{1}\\
 & \hphantom{=\varphi_{rest_{1},[2]}^{1}=}-\tfrac{b_{1}\cos(q^{2})\tan(q^{3})}{b_{2}}\left(q^{2}-v_{1}^{1}+2T\omega^{2}+T^{2}\left(a_{1}\sin(q^{2})+a_{2}\cos(q^{2})\right)\right)\,.
\end{align*}
In the second step, we have to shift the remaining component $\varphi_{rest_{1}}=\varphi^{2}$
of the flat output until it depends explicitly on the remaining input
$u^{2}$. This is the case for its fourth forward-shift
\begin{equation}
\delta^{4}(\varphi_{rest_{1}}^{1})=\varphi_{rest_{1},[4]}^{1}(q^{1},q^{2},q^{3},\omega^{1},\omega^{2},\omega^{3},v_{1}^{1},v_{1,[1]}^{1},v_{1,[2]}^{1},u^{2})\,.\label{eq:helicopter_varphi_K2}
\end{equation}
Because of $\mathrm{rank}(\partial_{u^{2}}\varphi_{rest_{1},[4]}^{1})=\dim(u^{2})=1$,
the procedure terminates with $\varphi_{2}=\varphi_{rest_{1}}$ and
$\kappa_{2}=K_{2}=4$. After introducing (\ref{eq:helicopter_varphi_K2})
as new input $v_{2}=v_{2}^{1}$, the forward-shifts of the flat output
(\ref{eq:helicopter_flat_output}) up to the orders $\kappa=(\kappa_{1},\kappa_{2})=(2,4)$
are given by
\begin{align}
y_{1,[0,\kappa_{1}-1]} & =\begin{bmatrix}\varphi_{1}^{1}\\
\varphi_{1,[1]}^{1}
\end{bmatrix}=\begin{bmatrix}q^{2}\\
q^{2}+T\omega^{2}
\end{bmatrix}\nonumber \\
y_{1,[\kappa_{1}]} & =\varphi_{1,[2]}^{1}=v_{1}^{1}\nonumber \\
y_{2,[0,\kappa_{2}-1]} & =\begin{bmatrix}\varphi_{2}^{1}\\
\varphi_{2,[1]}^{1}\\
\varphi_{2,[2]}^{1}\\
\varphi_{2,[3]}^{1}
\end{bmatrix}=\begin{bmatrix}q^{1}\\
q^{1}+T\omega^{1}\\
\varphi_{2,[2]}^{1}(q^{1},q^{2},q^{3},\omega^{1},\omega^{2},v_{1}^{1})\\
\varphi_{2,[3]}^{1}(q^{1},q^{2},q^{3},\omega^{1},\omega^{2},\omega^{3},v_{1}^{1},v_{1,[1]}^{1})
\end{bmatrix}\label{eq:helicopter_new coordinates}\\
y_{2,[\kappa_{2}]} & =\varphi_{2,[4]}^{1}=v_{2}^{1}\,.\nonumber 
\end{align}
Substituting (\ref{eq:helicopter_new coordinates}) as well as $y_{1,[3]}^{1}=v_{1,[1]}^{1}$
and $y_{1,[4]}^{1}=v_{1,[2]}^{1}$ into the parameterization (\ref{eq:helicopter_parameterization})
of the control input $u$ yields a feedback of the form
\begin{equation}
u=F_{u}\circ\phi(q^{1},q^{2},q^{3},\omega^{1},\omega^{2},\omega^{3},v_{1}^{1},v_{1,[1]}^{1},v_{1,[2]}^{1},v_{2}^{1})\label{eq:helicopter_linearizing_feedback}
\end{equation}
such that the input-output behaviour of the closed-loop system is
given by
\begin{equation}
\begin{array}{ccl}
y_{1,[2]}^{1} & = & v_{1}^{1}\\
y_{2,[4]}^{1} & = & v_{2}^{1}\,.
\end{array}\label{eq:helicopter_exactly_linearized}
\end{equation}
Thus, in contrast to (\ref{eq:helicopter_dynamic_linearization}),
we can actually use the second instead of the fourth forward-shift
of the first component of the flat output as a new input. Since the
linearizing feedback (\ref{eq:helicopter_linearizing_feedback}) does
not depend on backward-shifts of the system variables, it is a discrete-time
quasi-static feedback in the sense of \cite{ArandaKotta:2001} --
see the discussion after Corollary \ref{cor:quasistatic_linearization}.

For the exactly linearized system (\ref{eq:helicopter_exactly_linearized}),
the design of a tracking control is again straightforward. The control
law
\begin{align*}
v_{1}^{1} & =y_{1,[2]}^{1,d}-a_{1}^{1,1}(y_{1,[1]}^{1}-y_{1,[1]}^{1,d})-a_{1}^{1,0}(y_{1}^{1}-y_{1}^{1,d})\\
v_{2}^{1} & =y_{2,[4]}^{1,d}-\sum_{\beta=0}^{3}a_{2}^{1,\beta}(y_{2,[\beta]}^{1}-y_{2,[\beta]}^{1,d})
\end{align*}
results in the linear tracking error dynamics
\begin{align*}
e_{1,[2]}^{1}+a_{1}^{1,1}e_{1,[1]}^{1}+a_{1}^{1,0}e_{1}^{1} & =0\\
e_{2,[4]}^{1}+\sum_{\beta=0}^{3}a_{2}^{1,\beta}e_{2,[\beta]}^{1} & =0
\end{align*}
with a total order of $\#\kappa=2+4=6$ instead of $\#R=4+4=8$. The
system of equations (\ref{eq:final_equ_sys}) is given by
\begin{align*}
v_{1}^{1} & =y_{1,[2]}^{1,d}-a_{1}^{1,1}(\varphi_{1,[1]}^{1}-y_{1,[1]}^{1,d})-a_{1}^{1,0}(\varphi_{1}^{1}-y_{1}^{1,d})\\
v_{1,[1]}^{1} & =y_{1,[3]}^{1,d}-a_{1}^{1,1}(v_{1}^{1}-y_{1,[2]}^{1,d})-a_{1}^{1,0}(\varphi_{1,[1]}^{1}-y_{1,[1]}^{1,d})\\
v_{1,[2]}^{1} & =y_{1,[4]}^{1,d}-a_{1}^{1,1}(v_{1,[1]}^{1}-y_{1,[3]}^{1,d})-a_{1}^{1,0}(v_{1}^{1}-y_{1,[2]}^{1,d})\\
v_{2}^{1} & =y_{2,[4]}^{1,d}-\sum_{\beta=0}^{3}a_{2}^{1,\beta}(\varphi_{2,[\beta]}^{1}-y_{2,[\beta]}^{1,d})
\end{align*}
with the functions $\varphi_{1}^{1},\varphi_{1,[1]}^{1},\varphi_{2}^{1},\varphi_{2,[1]}^{1},\varphi_{2,[2]}^{1},\varphi_{2,[3]}^{1}$
according to (\ref{eq:helicopter_new coordinates}), and can be solved
from top to bottom for $v_{1}^{1},v_{1,[1]}^{1},v_{1,[2]}^{1}$, and
$v_{2}^{1}$ as a function of $q^{1},q^{2},q^{3},\omega^{1},\omega^{2},\omega^{3}$
and the reference trajectory $y_{[0,R]}^{d}$. Substituting the solution
into the linearizing feedback (\ref{eq:helicopter_linearizing_feedback})
finally yields a control law of the form
\begin{align*}
u^{1} & =\eta^{1}(q^{1},q^{2},q^{3},\omega^{1},\omega^{2},\omega^{3},y_{1,[0,4]}^{1,d},y_{2,[0,4]}^{1,d})\\
u^{2} & =\eta^{2}(q^{1},q^{2},q^{3},\omega^{1},\omega^{2},\omega^{3},y_{1,[0,4]}^{1,d},y_{2,[0,4]}^{1,d})\,.
\end{align*}

\section{Conclusion}

In this contribution we have investigated the exact linearization
of flat discrete-time systems. Since an exact linearization can always
be achieved by choosing the highest forward-shifts of the flat output
in (\ref{eq:flat_parameterization}) as new input $v=y_{[R]}$, the
point of departure of our considerations was the question whether
also lower-order forward-shifts $v=y_{[A]}$ with $A\leq R$ can be
used. Similar to the continuous-time case, this allows e.g. to achieve
a lower-order error dynamics for a subsequently designed tracking
control. Concerning the choice of a feasible new input $v=y_{[A]}$,
we have derived conditions which are formulated in terms of the linear
independence of certain differentials and can be checked in a straightforward
way. Furthermore, we have shown how the new input $v$ can be introduced
by a suitable dynamic feedback. For the practically quite important
case of flat outputs (\ref{eq:flat_output_xu}) which do not depend
on future values of the control input, we have shown how to construct
a minimal multi-index $\kappa$ such that $v=y_{[\kappa]}$ is a feasible
input and $\#\kappa\leq\#A$ for all other feasible inputs $v=y_{[A]}$.
Such an input $v=y_{[\kappa]}$ can be introduced by a feedback (\ref{eq:alternative_linearizing_feedback})
which depends only on the state $x$ as well as past values of the
system variables. This is particularly convenient for an implementation,
since past values of the system variables are available anyway from
past measurements or past inputs, which only need to be stored. Moreover,
we have shown that such an exact linearization can be used as a basis
for the design of a tracking control law which again only depends
on $x$ as well as past values of the system variables and the reference
trajectory. To illustrate our results, we have computed tracking control
laws for the discretized models of a wheeled mobile robot and a 3DOF
helicopter.

\bibliographystyle{apacite}
\bibliography{Bibliography_Bernd}

\appendix

\section{Supplements}

The following lemma addresses the relation between the functional
independence of functions and the linear independence of their differentials.
\begin{lem}
\label{lem:functional_independence}Consider a set of smooth functions
$g^{1},\ldots,g^{k}$ as well as another smooth function $h$ which
are all defined on the same manifold. Then
\begin{equation}
\mathrm{d}h\in\mathrm{span}\{\mathrm{d}g^{1},\ldots,\mathrm{d}g^{k}\}\label{eq:dh_in_span_dg}
\end{equation}
is equivalent to the existence of a smooth function $\psi:\mathbb{R}^{k}\mapsto\mathbb{R}$
such that locally 
\begin{align}
h & =\psi(g^{1},\ldots,g^{k})\label{eq:h_function_of_g}
\end{align}
holds identically. If the differentials $\mathrm{d}g^{1},\ldots,\mathrm{d}g^{k}$
are linearly independent, then the function $\psi$ is unique.
\end{lem}

\begin{proof}
Let $l\leq k$ denote the maximal number of linearly independent differentials
from the set $\{\mathrm{d}g^{1},\ldots,\mathrm{d}g^{k}\}$, and assume
that these differentials are given by $\mathrm{d}g^{1},\ldots,\mathrm{d}g^{l}$
(which is always possible by a renumbering). Then the functions $g^{1},\ldots,g^{l}$
can be introduced as (a part of the) coordinates
\[
z^{i}=g^{i}\,,\quad i=1,\ldots,l
\]
on the considered manifold. Moreover, by construction, also the functions
$g^{l+1},\ldots,g^{k}$ can depend only on the coordinates $z$. With
such coordinates, (\ref{eq:dh_in_span_dg}) is equivalent to
\begin{equation}
\mathrm{d}h\in\mathrm{span}\{\mathrm{d}z^{1},\ldots,\mathrm{d}z^{l}\}\,.\label{eq:dh_in_span_dz}
\end{equation}
Thus, the function $h$ can only depend on $z^{1},\ldots,z^{l}$,
and hence be written in original coordinates as
\[
h=\psi(g^{1},\ldots,g^{l})\,.
\]
Conversely, it is clear that
\[
h=\psi(z^{1},\ldots,z^{l},g^{l+1}(z),\ldots,g^{k}(z))
\]
implies (\ref{eq:dh_in_span_dz}) and hence (\ref{eq:dh_in_span_dg}).
In the case $l=k$, all functions $g^{1},\ldots,g^{k}$ can be introduced
as new coordinates, and no choice of $l$ independent ones is necessary.
Thus, the representation (\ref{eq:h_function_of_g}) is then unique.
\end{proof}

\end{document}